\pdfoutput=1
\documentclass[intlimits,a4paper,10pt,reqno]{elsarticle-mr}

\usepackage[%
   colorlinks=true,
   urlcolor=rltblue,       
   filecolor=rltgreen,     
   citecolor=rltgreen,     
   linkcolor=rltred,       
   bookmarks=true,
   unicode,
   plainpages=false,
   ]{hyperref}

\usepackage[operator,rose,frak]{paper_diening}
\usepackage[active]{srcltx}
\usepackage{amsmath,color}
\usepackage{amssymb,amsthm,amsxtra,amscd}
\usepackage{microtype}
\usepackage{esint}
\usepackage[applemac]{inputenc}

\definecolor{rltred}{rgb}{0.75,0,0}
\definecolor{rltgreen}{rgb}{0,0.5,0}
\definecolor{rltblue}{rgb}{0,0,0.75}

\newtheorem{Def}[equation]{Definition}
\newtheorem{Sa}[equation]{Theorem}
\newtheorem{Lem}[equation]{Lemma}
\newtheorem{Prop}[equation]{Proposition}
\newtheorem{Bem}[equation]{Remark}
\newtheorem{Kor}[equation]{Corollary}
\newtheorem{Vss}[equation]{Assumption}

\newenvironment{Bew}{\begin{proof}[Proof]}{\end{proof}}

\newcommand{\R}{\mathbb{R}}
\newcommand{\E}{\mathbf{E}}
\newcommand{\D}{\mathbf{D}}
\newcommand{\Rr}{\mathbf{R}}
\newcommand{\w}{\boldsymbol\omega}

\newcommand{\x}{\boldsymbol\xi}

\def\dint{\fint}

\DeclareMathOperator{\Div}{\operatorname{div}}

\newcommand{\Ss}{\mathbf{S}\big(\mathbf{D}\mathbf{v},\mathbf{R}(\mathbf{v},\boldsymbol\omega),\mathbf{E}\big)}

\newcommand{\Sn}{\mathbf{S}\big(\mathbf{D}\mathbf{v}^n,\mathbf{R}(\mathbf{v}^n,\boldsymbol\omega^n),\mathbf{E}\big)}

\newcommand{\N}{\mathbf{N}(\nabla\boldsymbol\omega,\mathbf{E})}

\newcommand{\Nn}{\mathbf{N}(\nabla\boldsymbol\omega^n,\mathbf{E})}

\newcommand{\anti}{{\ensuremath{\mathrm{skew}}}}
\newcommand{\bell}{\boldsymbol{\ell}}

\newcommand{\Lp}{L^p(\Omega)}

\newcommand{\XGqpmu}{X^{q(\cdot),p(\cdot)}_{{\nabla}}(\Omega;\sigma)}
\newcommand{\XGqpmuo}{{\mathaccent23 X}^{q(\cdot),p(\cdot)}_{{\nabla}}(\Omega;\sigma)}

\newcommand{\XGpEo}{{\mathaccent23 X}^{p^-,p(\cdot)}_{{\nabla}}(\Omega;\vert\mathbf{E}\vert^2)}
\newcommand{\XGqp}{X^{q(\cdot),p(\cdot)}_{{\nabla}}(\Omega)}
\newcommand{\XGqpo}{{\mathaccent23 X}^{q(\cdot),p(\cdot)}_{{\nabla}}(\Omega)}

\newcommand{\XDqp}{X^{q(\cdot),p(\cdot)}_{\mathbf{D}}(\Omega)}
\newcommand{\XDqpo}{{\mathaccent23 X}^{q(\cdot),p(\cdot)}_{\mathbf{D}}(\Omega)}
\newcommand{\XDp}{X^{p^-,p(\cdot)}_{\mathbf{D}}(\Omega)}
\newcommand{\XDpo}{{\mathaccent23 X}^{p^-,p(\cdot)}_{\mathbf{D}}(\Omega)}
\newcommand{\VDqpo}{{\mathaccent23 X}^{q(\cdot),p(\cdot)}_{\mathbf{D},\divo}(\Omega)}
\newcommand{\VDpo}{{\mathaccent23 X}^{p^-,p(\cdot)}_{\mathbf{D},\divo}(\Omega)}

\newcommand{\Io}{\int_{\Omega}}

\newcommand{\Lpwx}{L^{p(\cdot)}(\Omega;\sigma)}

\numberwithin{equation}{section}

\begin{document}
\begin{frontmatter}

\title{Existence of steady solutions for a model for micropolar electrorheological fluid flows with not globally $\log$--Hölder continuous shear exponent}

\author[ak]{Alex Kaltenbach}
\ead{alex.kaltenbach@mathematik.uni-freiburg.de}

\author[mr]{\corref{cor1}Michael R\r u\v zi\v cka}
\ead{rose@mathematik.uni-freiburg.de}

\cortext[cor1]{Corresponding author}

\address[mr]{Institute of Applied Mathematics,
  Albert--Ludwigs--University Freiburg, Ernst--Zermelo--Str.~1, D-79104 Freiburg,
  GERMANY.}  

\address[ak]{Institute of Applied Mathematics,
	Albert--Ludwigs--University Freiburg, Ernst--Zermelo--Str.~1, D-79104 Freiburg,
	GERMANY.}

\begin{abstract}
  In this paper, we study the existence of weak solutions to a steady
  system that describes the motion of a micropolar electrorheological
  fluid. The constitutive relations for the stress tensors belong to
  the class of generalized Newtonian fluids.  The analysis of this
  particular problem leads naturally to weighted variable exponent
  Sobolev spaces. We establish the existence of
  solutions for a material function $\hat p$ that is
  $\log$--Hölder continuous and an electric field $\bE$ for that 
  $\vert \bfE\vert^2$ is bounded and smooth. Note that these conditions do not imply that
  the variable shear exponent $p=\hat p\circ\vert \bfE\vert^2$ is
  globally $\log$--Hölder continuous.
\end{abstract}


\end{frontmatter}

\section{Introduction}\label{introduction}
In this paper we establish the existence of solutions of the
system\footnote{We denote by $\bfvarepsilon$ the isotropic third order
	tensor and by $\bfvarepsilon:\mathbf{S}$ the vector with the
	components $\vep_{ijk}S_{jk}$, $i=1,\ldots,d$, where the summation
	convention over repeated indices is used.}
\begin{align}
	\begin{aligned}\label{NS}
		-\Div \mathbf{S}+\Div(\mathbf{v}\otimes\mathbf{v})
		+\nabla\pi&=\mathbf{f} &&\text{in} \ \Omega\,,
		\\
		\Div \mathbf{v}&=0 &&\text{in} \ \Omega\,,
		\\
		-\Div\mathbf{N}+\Div(\boldsymbol{\omega}\otimes
		\mathbf{v})&=\boldsymbol{\ell}-\bfvarepsilon:\mathbf{S}
		&\quad&\text{in} \ \Omega\,,\\
		\mathbf{v}=\bfzero\,, \quad \boldsymbol\omega &=\bfzero
		&&\text{on}\ \partial\Omega \,.
	\end{aligned}
\end{align}
Here, $\Omega\subseteq \setR^d$, $d\ge 2$, is a bounded domain. The three
equations in \eqref{NS} represent the balance of momentum, mass and angular
momentum for an incompressible, micropolar electrorheological
fluid.  In it, $\bv$ denotes the velocity, $\w$ the
micro-rotation, $\pi$ the pressure, $\bS$ the mechanical extra
stress tensor, $\bN$ the couple stress tensor, $\bell$ the
electromagnetic couple force, ${\ff=\tilde \ff + \chi^E\divo (\bE
	\otimes \bE)}$ the body force, where $\tilde \ff$ is the mechanical
body force, $\chi^E$ the dielectric susceptibility and $\bE$ the
electric field. The electric field $\bE$ solves the quasi-static
Maxwell's equations
\begin{align}
	\begin{aligned}\label{maxwell}
		&& \Div \mathbf{E}&=0 &&\text{in}\ \Omega\,,
		\\
		&& \curl \mathbf{E}&=\bfzero &&\text{in}\ \Omega\,,
		\\
		&& \mathbf{E}\cdot \mathbf{n}&=\mathbf{E}_0\cdot \mathbf{n}
		&\quad&\text{on}\ \partial\Omega\,,
	\end{aligned}
\end{align}
where $\mathbf{n}$ is the outer normal vector field of
$\partial \Omega$ and $\mathbf{E}_0$ is a given~electric~field. The
system \eqref{NS}, \eqref{maxwell} is the steady version of a model
derived in \cite{win-r}, which generalizes previous models of
electrorheological fluids in \cite{RR2},
\cite{rubo}.~The~model~in~\cite{win-r} contains a more realistic
description of the dependence of the electrorheological effect on the
direction of the electric field. Since Maxwell's equations
\eqref{maxwell} are separated from the balance laws \eqref{NS} and due
to the well developed \mbox{mathematical} theory for Maxwell's equations (cf.~Section~\ref{sec:E}), we can
view the electric~field~$\bE$ with appropriate properties as a given
quantity in \eqref{NS}. As a consequence,~we concentrate in this paper on the investigation of the mechanical properties~of~the
electrorheological fluid governed by \eqref{NS}.

A representative example for a constitutive
relation for the stress tensors in \eqref{NS} reads, e.g., (cf.~\cite{win-r},
\cite{rubo})
\begin{align}
	\hspace*{-1mm}  
	\begin{aligned}\label{eq:SN-ex}
		\mathbf{S}&=(\alpha_{31}+\alpha_{33}\vert\mathbf{E}\vert^2)
		(1+\vert\mathbf{D}\vert)^{p-2}\mathbf{D}+ \alpha_{51}
		(1+\vert\mathbf{D}\vert)^{p-2}\big (\mathbf{D} \bE \otimes \bE +
		\bE \otimes \bD\bE \big)\hspace*{-5mm}
		\\
		&\quad +
		\alpha_{71}\vert\mathbf{E}\vert^2(1+\vert\mathbf{R}\vert)^{p-2}
		\mathbf{R} + \alpha_{91} (1+\vert\mathbf{R}\vert)^{p-2}\big
		(\mathbf{R} \bE \otimes \bE + \bE \otimes \bR\bE \big)\,,
		\\
		\mathbf{N}&=(\beta_{31}+\beta_{33}\vert\mathbf{E}\vert^2)
		(1+\vert\nabla\boldsymbol\omega\vert)^{p-2}\nabla\boldsymbol\omega
		\\
		&\quad + \beta_{51}(1+\vert\nabla \w\vert)^{p-2}\big ((\nabla \w
		)\bE \otimes \bE + \bE \otimes (\nabla \w)\bE \big)\,,
	\end{aligned}
\end{align}
with material constants $\alpha_{31},\alpha_{33},\alpha_{71},\beta_{33}>
0$ and $\beta_{31}\ge 0$ and a shear exponent $p=\hat p \circ
\abs{\bE}^2$, where $\hat p$ is a material function. 
In \eqref{eq:SN-ex}, we employed the common notation\footnote{Here,
	$\bfepsilon:\bv$ denotes the tensor with components $\vep_{ijk}v_k$,
	$i,j=1,\ldots,d$.} $\bD = (\nabla \bv)^\sym$ and
${\bR=\bR(\bv,\w):= (\nabla \bv)^\anti +\bfvarepsilon :\w }$. 

Micropolar fluids have been introduced by Eringen in the sixties
(cf.~\cite{eringen-book}).~A model for electrorheological fluids was
proposed in \cite{RR1},~\cite{RR2},~\cite{rubo}.~While~there~exist
many investigations of micropolar fluids or
electrorheological~fluids (cf.~\cite{Lukaszewicz},~\cite{rubo}), there
exist to our knowledge no mathematical investigations of
steady~motions~of micropolar electrorheological fluids except the PhD
thesis \cite{frank-phd},~the~diploma~thesis \cite{weber-dipl}, the
paper \cite{erw} and the recent result \cite{kr-micro}. Except for
the latter contribution these investigations only treat the
case of constant shear exponents.

For the existence theory of problems of similar type as \eqref{NS},
the Lipschitz truncation technique (cf.~\cite{fms2}, \cite{dms}) has
proven to be very~powerful.~This~method is available in the setting of
Sobolev spaces (cf.~\cite{fms}, \cite{dms}, \cite{john}), variable
exponent Sobolev spaces (cf.~\cite{dms}, \cite{john}), solenoidal
Sobolev spaces (cf.~\cite{bdf}), Sobolev spaces with
Muckenhoupt~weights~(cf.~\cite{erw}) and
functions~of~bounded~\mbox{variation} (cf.~\cite{BDG19}).  Since, in
general, $\vert \bE\vert^2$ does not belong to~the~correct Muckenhoupt
class, the results in \cite{erw} are either sub-optimal with respect
to the lower bound for the shear~exponent~$p$ or require additional
assumptions on the electric field~$\bE$. These deficiencies are
overcome in \cite{kr-micro} by an thorough localization of the
arguments. Moreover, \cite{kr-micro} contains the first treatment of
the full model for micropolar electrorheological fluids in weighted
variable exponent spaces under the assumption that the shear exponent
is globally $\log$--H\"older continuous. The present paper relaxes this
condition and shows existence of solutions under the only assumption
that the electric field $\bE$ is bounded and smooth.

\smallskip \textit{This paper is organized as follows:} In Section \ref{veroeffentlichung1}, we
introduce the functional setting for the treatment of the variable
exponent weighted case, and
collect auxiliary results. Then, Section~\ref{sec:E} is devoted to the analysis of the
electric field, while Section
\ref{sec:stab} is devoted to the weak stability of the stress
tensors. Eventually, in Section \ref{sec:main}, we deploy the Lipschitz
truncation technique to prove the existence of weak solutions of
\eqref{NS}, \eqref{maxwell}.

\section{Preliminaries}\label{veroeffentlichung1}

\subsection{Basic notation and standard function spaces}

We employ the customary Lebesgue spaces $\Lp$, $1\leq p\leq\infty$, and
Sobolev spaces $W^{1,p}(\Omega)$, $1\leq p\leq\infty$, where
$\Omega\subseteq \R^d$, $d\in \mathbb{N}$, is a bounded domain. We
denote by $\Vert\cdot\Vert_p$ the norm in $\Lp$ and by
$\Vert\cdot\Vert_{1,p}$ the norm~in~$W^{1,p}(\Omega)$. 
The space $W^{1,p}_0(\Omega)$, $1\le p <\infty$, is
defined as the completion of $C_0^\infty(\Omega)$ with respect to the
gradient norm $\Vert\nabla \cdot\Vert_{p}$, while the space
$W^{1,p}_{0,\divo}(\Omega)$, $1\le p <\infty$, is the closure of
$C_{0,\textup{div}}^\infty(\Omega):=\{\bu\in C_0^\infty(\Omega)\fdg
\Div \bu=0\}$ with respect to the gradient norm
$\Vert\nabla
\cdot\Vert_{p}$.~For~a~bounded~Lipschitz~domain~${G\subseteq \R^d} $,
we define $W^{1,\infty}_0(G)$ as the subspace of functions
$u \in W^{1,\infty}(G)$ having a vanishing trace, i.e.,
$u|_{\partial G}=0$. We use small boldface letters, e.g.,~$\bv$, to
denote vector-valued functions and capital boldface letters,
e.g.,~$\bS$, to denote~\mbox{tensor-valued}~functions\footnote{The
	only exception of this is the electric vector field which is denoted
	as usual by $\bE$.}. However, we do not distinguish between scalar,
vector-valued and tensor-valued function spaces in the notation. The
standard scalar product between vectors is denoted by $\bv\cdot \bu$,
while the standard scalar product between tensors is
denoted~by~$\bA:\bB$. For a normed linear vector space $X$, we denote
its topological dual space by $X^*$.  Moreover, we employ the notation
$\langle u,v\rangle:=\Io uv\,dx$, whenever the right-hand side is
well-defined.  We denote by $\vert M\vert$ the $d$--dimensional
Lebesgue measure of a measurable set $M$. The mean value of a locally
integrable function $u\in L^1_{\loc}(\Omega)$ over a measurable set
$M\subseteq\Omega$ is denoted by
$\dint_M u\,dx:=\frac{1}{\vert M\vert} \int_M u\,dx$. By
$L^p_0(\Omega)$ and $C^\infty_{0,0}(\Omega)$, resp., we denote the
subspace of $L^p(\Omega)$~and~$C^\infty_{0}(\Omega)$,~resp.,
consisting of all functions $u$ with
vanishing~mean~value,~i.e.,~${\dint_\Omega u\,dx =0}$.

\subsection{Weighted variable exponent Lebesgue and Sobolev spaces}

In this section, we will give a brief introduction into weighted~variable~\mbox{exponent} Lebesgue and Sobolev spaces.

Let $\Omega\subseteq \mathbb{R}^d$, $d\in \mathbb{N}$, be an open set
and $p:\Omega\to [1,\infty)$ a \mbox{measurable~function}, called
variable exponent in $\Omega$. By $\mathcal{P}(\Omega)$, we denote the
set of all variable exponents. For $p\in \mathcal{P}(\Omega)$, we denote by
${p^+:=\textup{ess\,sup}_{x\in \Omega}{p(x)}}$~and~${p^-:=\textup{ess\,inf}_{x\in \Omega}{p(x)}}$ its limit exponents. By $\mathcal{P}^{\infty}(\Omega):=\{p\in\mathcal{P}(\Omega)\fdg p^+<\infty\}$,
we denote the set of all bounded variable exponents. 

A weight $\sigma$ on $\R^d$
is a locally integrable function satisfying $0<\sigma<\infty$
a.e.\footnote{If not stated otherwise, a.e.~is meant with respect to
	the Lebesgue measure.}. To each weight $\sigma$, we associate a Radon
measure $\nu_{\sigma}$ defined via $\nu_{\sigma} (A):=\int _A \sigma\, dx$
for every measurable set $A\subseteq \mathbb{R}^d$.

For an exponent $p\in \mathcal{P}^{\infty}(\Omega)$ and a weight $\sigma$, the weighted variable exponent Lebesgue space $\Lpwx$ consists of all measurable~functions ${u:\Omega\to \mathbb{R}}$, i.e., $u\in \mathcal{M}(\Omega)$, for which the modular 
\begin{align*}
	\rho_{p(\cdot),\sigma}(u):=\int_{\Omega}{\vert u(x)\vert^{p(x)}\,d\nu_{\sigma}(x)}:=\int_{\Omega}{\vert u(x)\vert^{p(x)}\sigma(x)\,dx}
\end{align*}
is finite, i.e., we have that $\Lpwx:=\{u\in \mathcal{M}(\Omega)\fdg \sigma^{1/p(\cdot)}u\in L^{p(\cdot)}(\Omega)\}$.~We equip $\Lpwx$ with the Luxembourg norm 
\begin{align*}
	\|u\|_{p(\cdot),\sigma}:=\inf\big\{\lambda > 0\fdg \rho_{p(\cdot),\sigma}(u/\lambda)\leq 1\big\}\,,
\end{align*}
which turns $\Lpwx$ into a separable Banach space (cf. \cite[Thm.~3.2.7~\& Lem.~3.4.4]{lpx-book}). In addition, if ${\sigma =1}$~a.e.~in~$\Omega$, then we employ the abbreviations $L^{p(\cdot)}(\Omega):=\Lpwx$, $\rho_{p(\cdot)}(u):=\rho_{p(\cdot),\sigma}(u)$ and ${\|u\|_{p(\cdot)}:=\|u\|_{p(\cdot),\sigma}}$ for every ${u\in L^{p(\cdot)}(\Omega)}$. The identity $\rho_{p(\cdot),\sigma}(u)\!=\!\rho_{p(\cdot)}(u\sigma^{1/p(\cdot)})$~implies that
\begin{align}
	\|u\|_{p(\cdot),\sigma}=\|u\sigma^{1/p(\cdot)}\|_{p(\cdot)}\label{eq:mo-no}
\end{align}
for all $u\in \Lpwx$. If $p\in \mathcal{P}^{\infty}(\Omega)$, in addition, satisfies $p^->1$, then~$\Lpwx$ is reflexive (cf. \cite[Thm. 3.4.7]{lpx-book}).  The dual space $(\Lpwx)^*\!$ can be
identified with respect to $\skp{\cdot}{\cdot}$ with $
L^{p'(\cdot)}(\Omega;\sigma')$, where
$\sigma':=\sigma^{\smash{\frac{-1}{p(\cdot)-1}}}$. The identity
\eqref{eq:mo-no} and Hölder's inequality in variable exponent Lebesgue
spaces (cf. \cite[Lem.~3.2.20]{lpx-book}) yield for every $u\in \Lpwx$
and  ${v \in L^{p'(\cdot)}(\Omega;\sigma')}$, there holds 
\begin{equation*}
	\vert \skp{u}{v}\vert\le 2\,\norm{u}_{p(\cdot),\sigma} \norm{v}_{p'(\cdot),\sigma'}\,.
\end{equation*}

The relation between the modular and the norm is clarified by the
following lemma, which is called norm-modular unit ball property.
\begin{Lem}
	\label{lem:unit_ball_px}
	Let $\Omega\subseteq \R^d$, ${d\in
		\mathbb{N}}$,~be~open and let $p\in
	\mathcal{P}^{\infty}(\Omega)$. Then, we have for
	any $u \in \Lpwx$:
	\begin{enumerate}[(i)]
		\item $\norm{u}_{p(\cdot),\sigma} \leq 1$ if and only if   $\rho_{p(\cdot),\sigma} (u) \leq 1$. 
		\item \label{itm:unit_ball2pxa} If $\norm{u}_{p(\cdot),\sigma} \leq 1$, then
		$\rho _{p(\cdot),\sigma} (u) \leq
		\norm{u}_{p(\cdot),\sigma}$.\label{itm:unit_ball2pxb}
		\item If $ 1< \norm{u}_{p(\cdot),\sigma}$, then
		$\norm{u}_{p(\cdot),\sigma} \leq \rho _{p(\cdot),\sigma} (u)$.
		\item $\smash{\norm{u}_{p(\cdot),\sigma}^{p^-} -1 \le \rho _{p(\cdot),\sigma} (u) \leq
		\norm{u}_{p(\cdot),\sigma}^{p^+} +1}$.
	\end{enumerate}
\end{Lem}
\begin{Bew}
	See \cite[Lem.~3.2.4 \& Lem.~3.2.5]{lpx-book}.
\end{Bew}

The following generalization of a classical result (cf.~\cite{ggz}) is
very useful~in~the identification of limits.
\begin{Sa}\label{pfastue}
	Let $\Omega\subseteq\R^d$,
	$d\in\mathbb{N}$, be bounded,
	$\sigma\in L^\infty(\Omega)$ a weight and $p\in \mathcal{P}^{\infty}(\Omega)$. Then, for a sequence
	$(u_n)_{n\in \mathbb{N}}\subseteq L^{p(\cdot)}(\Omega;\sigma)$
	from
	\begin{enumerate}
		\item [{\rm (i)}]$\lim\limits_{n\to\infty} u_n=v$ $\nu_{\sigma}$--a.e. in
		$\Omega$,
		\item [{\rm (ii)}] $u_n\rightharpoonup u$ in $L^{p(\cdot)}(\Omega;\sigma)$ $(n\to \infty)$,
	\end{enumerate}
	it follows that $u=v$ in $L^{p(\cdot)}(\Omega;\sigma)$.
\end{Sa}

\begin{Bew}
	For a proof in the case of constant exponents, we refer to \cite[Thm.~13.44]{Hew-strom-65}. However, because $L^{p(\cdot)}(\Omega;\sigma)\hookrightarrow L^1(\Omega;\sigma)$ for both ${p\in \mathcal{P}^{\infty}(G)}$ and $\sigma\in L^\infty(G)$, the non-constant case follows from the constant case.
\end{Bew}

Let us now introduce variable exponent Sobolev spaces in the weighted
and unweighted case. Let us start with the unweighted case. Due to $L^{p(\cdot)}(\Omega)\hookrightarrow L^1_{\loc}(\Omega)$, we
can define the variable exponent~Sobolev~space
$W^{1,p(\cdot)}(\Omega)$ as the subspace of $L^{p(\cdot)}(\Omega)$ consisting of all
functions $u\in L^{p(\cdot)}(\Omega)$ whose distributional gradient
satisfies
$\nabla u\hspace*{-0.1em}\in\hspace*{-0.1em}
L^{p(\cdot)}(\Omega)$. The norm
${\|\cdot\|_{1,p(\cdot)}\hspace*{-0.1em}:=\hspace*{-0.1em}\|\cdot\|_{p(\cdot)}\hspace*{-0.1em}+\hspace*{-0.1em}\|\nabla\cdot\|_{p(\cdot)}}$
turns $W^{1,p(\cdot)}(\Omega)$ into a separable Banach space
(cf. \cite[Thm. 8.1.6]{lpx-book}). Then, we define the space
$\smash{W^{1,p(\cdot)}_0(\Omega)}$~as~the~closure of
$C_0^\infty(\Omega)$ in
$\smash{W^{1,p(\cdot)}\hspace*{-0.1em}(\Omega)}$, while
$\smash{W^{1,p(\cdot)}_{0,\divo}(\Omega)}$ is~the~closure of
$C_{0,\divo}^\infty(\Omega)$ in $\smash{W^{1,p(\cdot)}(\Omega)}$. If
$p\in \mathcal{P}^{\infty}(\Omega)$, in addition,
satisfies~${p^->1}$,~then~the spaces $\smash{W^{1,p(\cdot)}(\Omega)}$,
$\smash{W^{1,p(\cdot)}_0(\Omega)}$ and
$\smash{W^{1,p(\cdot)}_{0,\divo}(\Omega)}$ are reflexive
(cf. \cite[Thm. 8.1.6]{lpx-book}).

Note that the velocity field $\bfv: \Omega \to \mathbb{R}^d$ solving
\eqref{NS}, in view of the properties of the extra stress tensor
(cf.~Assumption \ref{VssSpx}), necessarily satisfies
$\bD\bfv\in L^{p(\cdot)}(\Omega)$. Even though we have that $\bfv=0$
on $\partial\Omega$, we cannot resort to Korn's inequality in the
setting of variable exponent Sobolev spaces
(cf.~\cite[Thm.~14.3.21]{lpx-book}), since we do not assume that
$p=\hat p\circ \vert \bfE\vert^2\in \mathcal{P}^{\infty}(\Omega)$ is
globally $\log$--Hölder continuous.~However, if we switch by means of
Hölder's inequality from the variable exponent
$p\in \mathcal{P}^{\infty}(\Omega)$ to its lower bound $p^-$, for
which Korn's inequality is available, also using Poincar\'e's
inequality, we can expect that~a~solution~$\bfv$~of~\eqref{NS}~satisfies~$\smash{\bfv\in
  L^{p^-}(\Omega)}$. 
Thus, the natural energy space for the velocity possesses a different
integrability for the function and its symmetric gradient. This motivates
the introduction of the following variable exponent function spaces.
\begin{Def}
  Let  $\Omega\subseteq\R^d$, $d\in \mathbb{N}$, be open and
  $q,p\in\mathcal{P}^{\infty}(\Omega)$.  For $\bu\in
  C^\infty(\Omega)$, we define
  \begin{gather*}
    \Vert\bu\Vert _{\smash{\XDqp}}:=
    \|\bu\|_{q(\cdot)} +\|\bD\bu\|_{p(\cdot)}\,.
  \end{gather*}
  The space $\smash{\XDqp}$ is defined as the completion of
  \begin{align*}
    \mathcal{V}^{q(\cdot),p(\cdot)}_{\mathbf{D}}:=\big\{\bu
    \in C^\infty(\Omega)\fdg\|\bu\|_{q(\cdot)}+\|\bD\bu\|_{p(\cdot)}<\infty \big\}
  \end{align*}
  with respect to $\|\cdot\|_{\smash{\XDqp}}$. 
\end{Def}

\begin{Prop}
		Let  $\Omega\subseteq\R^d$, $d\in \mathbb{N}$, be open and $q,p\in\mathcal{P}^{\infty}(\Omega)$. Then,~the~space $\smash{\XDqp}$ is a separable Banach space, which is reflexive if $q^-,p^->1$.
\end{Prop}

\begin{proof}
Clearly, $\|\cdot\|_{\smash{\XDqp}}$ is a norm  and, thus, $\XDqp$, by~\mbox{definition}, is a Banach space.
For the separability and reflexivity, we first observe that
$\Pi:\smash{\XDqp}\to L^{q(\cdot)}(\Omega)^d\times
L^{p(\cdot)}(\Omega)^{d\times d}$ is an isometry and, thus, an
isometric isomorphism into its range $R(\Pi)$. Hence, $R(\Pi)$
inherits the separability and reflexivity from
$L^{q(\cdot)}(\Omega)^d\times L^{p(\cdot)}(\Omega)^{d\times d}$, and
by virtue of the isometric isomorphism $\smash{\XDqp}$ as well. 
\end{proof}

\begin{Def}
	Let  $\Omega\subseteq\R^d$, $d\in \mathbb{N}$, be open and $q,p\in\mathcal{P}^{\infty}(\Omega)$. Then,~we~define the spaces
	\begin{align*}
		\smash{\XDqpo:=\overline{C_0^\infty(\Omega)}^{\smash{\|\cdot\|_{\smash{\XDqp}}}},\quad 	\VDqpo:=\overline{C_{0,\textup{div}}^\infty(\Omega)}^{\smash{\|\cdot\|_{\smash{\XDqp}}}}\,.}
	\end{align*}
\end{Def}

For the treatment of the micro-rotation $\bomega$, we also need  weighted
variable exponent Sobolev spaces. In analogy with \cite[Ass.~2.2]{kr-micro}, we make the following assumption on the weight $\sigma$.

\begin{Vss}\label{weightp(x)}
  Let $\Omega\subseteq\R^d$, $d\in \mathbb{N}$, be an open set and 
  $q,p\in\mathcal{P}^{\infty}(\Omega)$. The weight
  $\sigma$ is admissible, i.e., if a sequence ${(\bvarphi_n)_{n\in
    \mathbb{N}}\subseteq C^\infty(\Omega)}$  and ${\bT \in \Lpwx}$
  satisfy  ${\int_\Omega\vert\bvarphi_n(x)\vert^{q(x)}\sigma(x)\,dx\!\to \! 0}$
  $(n\!\to\!\infty)$ and
  ${\int_\Omega\vert\nabla\bvarphi_n(x)\!-\!\bT(x)\vert^{p(x)}\sigma(x)\,dx\!\to \!
  0}$ $(n\to\infty)$, then it follows that $\bT=\mathbf{0} $ in
  $\Lpwx$.
\end{Vss}

\begin{Bem}
	If $\sigma \in C^0(\Omega)$, then the same argumentation as in \cite[Rem.~2.3~(ii)]{kr-micro}  shows that  Assumption~\ref{weightp(x)} is satisfied for every $q,p\in \mathcal{P}^{\infty}(\Omega)$.
\end{Bem}


\begin{Def}
  Let $\Omega \subseteq \R^d$, $d\in \mathbb{N}$, be open and let
  $\sigma$ satisfy Assumption \ref{weightp(x)} for
  $q,p \in \mathcal{P}^{\infty}(\Omega)$.  For $\bu\in
  C^\infty(\Omega)$, we define
  \begin{gather*}
    \Vert\bu\Vert_{q(\cdot),p(\cdot),\sigma}:=
  \|\bu\|_{q(\cdot),\sigma} +\|\nabla\bu\|_{p(\cdot),\sigma}\,.
  \end{gather*}
  The weighted variable
  exponent Sobolev space $\smash{\XGqpmu}$ is defined as the
  completion of
  \begin{align*}
    \smash{\mathcal{V}^{q(\cdot),p(\cdot)}_{\nabla,\sigma}:=\big\{\bu\in
    C^\infty(\Omega)\fdg 
    \Vert\bu\Vert_{q(\cdot),p(\cdot),\sigma}<\infty\big\}}
  \end{align*}
  with respect to
  $\Vert\cdot\Vert_{q(\cdot),p(\cdot),\sigma}$. 
\end{Def}

In other words, $\bfw \in \XGqpmu$ if and only if
$\bw\in L^{q(\cdot)}(\Omega;\sigma)$ and there is a tensor field
$\bT\in \Lpwx$ such that for some sequence
$(\boldsymbol\varphi_n)_{n\in \mathbb{N}}\subseteq C^\infty(\Omega)$
holds
$\Io\vert\boldsymbol\varphi_n-\bw\vert^{q(\cdot )}\sigma\,dx\to 0$
$(n\to\infty)$ and
${\Io\vert\nabla\boldsymbol\varphi_n-\bT\vert^{p(\cdot)}\sigma\,dx\to
  0}$~${(n\to\infty)}$.  Assumption \ref{weightp(x)} implies that
$\bT $ is a uniquely defined function in $\Lpwx$~and~we, thus, set
$\smash{\hat{\nabla}}\bw:=\bT$ in $\Lpwx$. Note that
$\smash{W^{1,p(\cdot)}(\Omega)\hspace*{-0.1em}=\hspace*{-0.1em}X^{p(\cdot),p(\cdot)}_{\nabla}(\Omega;\sigma)}$~if~${\sigma=
    1}$ a.e. in $\Omega$ with $\smash{\nabla \bw =\hat\nabla \bw}$ for
  all $\smash{\bw\in W^{1,p(\cdot)}(\Omega)}$. However, in general,
  $\smash{\hat{\nabla}\bw}$ and the usual distributional
  gradient $\nabla \bw$ do not coincide.

\begin{Sa}
	Let  $\Omega\hspace*{-0.1em}\subseteq\hspace*{-0.1em}\R^d$,
        $d\hspace*{-0.1em}\in\hspace*{-0.1em} \mathbb{N}$, be open and
        let $\sigma$ satisfy Assumption~\ref{weightp(x)}. Then, 
	 $\smash{\XGqpmu} $ is a
	separable Banach space, which is reflexive if ${q^-,p^->1}$.
\end{Sa}

\begin{proof}
	The space $\XGqpmu$ is a Banach space by definition. So,
        it~is~left~to prove that it is  separable, and if, in
        addition, $q^-,p^->1$ it is reflexive. To this end, we first note that
	\begin{align}
		\|\bw\|_{q(\cdot),p(\cdot),\sigma}=\|\bw\|_{q(\cdot),\sigma}+\|\hat\nabla \bw\|_{p(\cdot),\sigma}\label{isometry}
	\end{align}
	for all $\bw\hspace*{-0.1em}\in\hspace*{-0.1em} \smash{\XGqpmu}$. In fact, for any $\bw\hspace*{-0.1em}\in\hspace*{-0.1em} \smash{\XGqpmu}$, by~definition,~there is a sequence $(\boldsymbol\varphi_n)_{n\in \mathbb{N}}\subseteq{\smash{\mathcal{V}^{q(\cdot),p(\cdot)}_{\sigma}}}$ such that $\boldsymbol\varphi_n\to \bw$~in~$L^{q(\cdot)}(\Omega;\sigma)$~${(n\to\infty)}$ and $\nabla\boldsymbol\varphi_n\hspace*{-0.1em}\to \hspace*{-0.1em}\hat\nabla \bw$ in $\Lpwx$ $(n\hspace*{-0.1em}\to\hspace*{-0.1em} \infty)$ and ${\|\boldsymbol\varphi_n\|_{q(\cdot),p(\cdot),\sigma}\hspace*{-0.1em}=\hspace*{-0.1em}
	\|\boldsymbol\varphi_n\|_{q(\cdot),\sigma}
	+\|\nabla\boldsymbol\varphi_n\|_{p(\cdot),\sigma}}$ for all $n\hspace*{-0.1em}\in\hspace*{-0.1em} \mathbb{N}$. The limit $n\hspace*{-0.1em}\to\hspace*{-0.1em} \infty$ yields~\eqref{isometry}~for~every~${\bw\hspace*{-0.1em}\in\hspace*{-0.1em} \XGqpmu}$.
	Then,  \eqref{isometry} implies that ${\Pi:\XGqpmu\to L^{q(\cdot)}(\Omega;\sigma)^d\times \Lpwx^{d\times d}}$, defined via $\Pi \bw\hspace*{-0.15em}:=\hspace*{-0.15em}(\bw,\!\hat\nabla \bw)^\top \hspace*{-0.15em}$ in $L^{q(\cdot)}(\Omega;\sigma)^d\times \Lpwx^{d\times d}$ for all $\bw\hspace*{-0.15em}\in\hspace*{-0.15em} \XGqpmu$, is an isometry,
	and, thus, an isometric isomorphism into its range $R(\Pi)$. 
	Hence, $R(\Pi)$ inherits the separability and reflexivity, if, in addition, $q^-,p^->1$, of $\smash{L^{q(\cdot)}(\Omega;\sigma)^d\times \Lpwx^{d\times d}}$ and~$\smash{\XGqpmu}$~as well.
\end{proof}
If, in addition,~${\sigma\in L^\infty(\Omega)}$, then we have that
$\smash{W^{1,p(\cdot)}(\Omega)\hookrightarrow
	{X}^{p(\cdot),p(\cdot)}_{{\nabla}}(\Omega;\sigma)}$ and ${\nabla \bw =\hat\nabla \bw}$ for
all $\smash{\bw\in W^{1,p(\cdot)}(\Omega)}$, which follows from the estimate
\begin{gather*}
	\smash{\|\bu\|_{p(\cdot),\sigma}=\|\bu\sigma^{1/p(\cdot)}\|_{{p(\cdot)}}\leq 2\,
		\|\sigma\|_{\infty}^{1/p^-}\|\bu\|_{{p(\cdot)}}} 
\end{gather*}
valid for every $\bu\in L^{p(\cdot)}(\Omega)$, in the same way as in
\cite[Sec.~1.9, Sec.~1.10]{heinonen}.

\begin{Def}
	Let  $\Omega\hspace*{-0.1em}\subseteq\hspace*{-0.1em}\R^d$, $d\hspace*{-0.1em}\in\hspace*{-0.1em} \mathbb{N}$, be open and let Assumption~\ref{weightp(x)}~be~\mbox{satisfied} for $q,p\in \mathcal{P}^{\infty}(\Omega)$. Then, we define the space
	\begin{align*}
		\smash{\XGqpmuo:=\overline{C_0^\infty(\Omega)}^{\|\cdot\|_{q(\cdot),p(\cdot),\sigma}}\,.}
	\end{align*}
\end{Def}
In the particular case $\smash{\sigma=1}$ a.e. in $\Omega$, we employ the abbreviations
\begin{gather*}
	\XGqp:=X^{q(\cdot),p(\cdot)}_{{\nabla}}(\Omega;1)\,,\quad\XGqpo:={\mathaccent23 X}^{q(\cdot),p(\cdot)}_{{\nabla}}(\Omega;1)\,.
\end{gather*}

The following local embedding result will play a key role for the
applicability of the Lipschitz truncation technique
in the proof  of the existence result in Theorem~\ref{thm:main4p(x)}.

\begin{Sa}\label{poincarep(x)}
	Let $\Omega\subseteq \mathbb{R}^d$, $d\ge 2$, be a bounded
        domain and let $p\in C^0(\Omega)$ satisfy
        $p^-\ge\frac{2d}{d+2}$. Then, for each
        $\Omega'\subset\subset\Omega$, there  exists a constant ${c(p,\Omega')>0}$  such that for every $\smash{\bfu\in \XDp}$, it holds $\smash{\bfu\in L^{p(\cdot)}(\Omega')}$ with
	\begin{align}
		\|\bfu\|_{L^{p(\cdot)}(\Omega')}\leq c(p,\Omega')\|\bfu\|_{\XDp}\,.\label{eq:poincarep(x)}
	\end{align}
	i.e., we have that $\XDp\hookrightarrow L^{p(\cdot)}_{\loc}(\Omega)$.
\end{Sa}

\begin{proof}
	The proof is postponed to the \ref{sec:app}.
\end{proof}

\subsection{$\log$--H\"older continuity and related results}\label{sec:log}

We say that a bounded exponent $p\in \mathcal P^\infty (G)$ is locally
$\log$--Hölder continuous, if there is a constant $c_1>0$ such that
for all $x,y\in G$
\begin{align*}
	\vert p(x)-p(y)\vert \leq \frac{c_1}{\log(e+1/\vert x-y\vert)}\,.
\end{align*}
We say that $p \in \mathcal P^\infty (G)$ satisfies the $\log$--Hölder decay condition, if there exist 
constants $c_2>0$ and $p_\infty\in \setR$ such that for all $x\in G$
\begin{align*}
	\vert p(x)-p_\infty\vert \leq\frac{c_2}{\log(e+1/\vert x\vert)}\,.
\end{align*} 
We say that $p$ is globally $\log$--Hölder continuous on $G$, if it is locally 
$\log$--Hölder continuous and satisfies the $\log$--Hölder decay condition. 
Then, the maximum ${c_{\log}(p):=\max\{c_1,c_2\}}$ is just called the $\log$--Hölder constant of $p$.
Moreover, we denote by $\mathcal{P}^{\log}(G)$ the set of 
globally $\log$--Hölder continuous
functions on $G$. 

$\log$--Hölder continuity  is a special modulus of continuity for variable exponents that is sufficient for the validity of the following results.

\begin{Sa}\label{bogp(x)}
	Let $G\subseteq \mathbb{R}^d\!$, $d\ge 2$, be a bounded Lipschitz
	domain. Then, there exists a linear operator
	$\mathcal{B}_G:C^{\infty}_{0,0}(G)\to C^\infty_0(G)$ which for all
	exponents $\smash {p\in \mathcal{P}^{\log}(G)}$ satisfying $p^->1$
	extends uniquely to a linear, bounded operator
	${\mathcal{B}_G:L^{p(\cdot)}_0(G)\to W^{1,p(\cdot)}_0(G)}$ such that
	$\|\mathcal{B}_Gu\|_{1,p(\cdot)}\leq c\,\|u\|_{p(\cdot)}$ and
	$\textup{div}\,\mathcal{B}_Gu = u$ for every
	$\smash{u\in L^{p(\cdot)}_0(G)}$.
\end{Sa}
\begin{Bew}
	See \cite[Thm.~2.2]{dr-nta}, \cite[Thm.~6.4]{dr-calderon}, \cite[Thm.~14.3.15]{lpx-book}.
\end{Bew}

\begin{Sa}\label{rellichp(x)}
	Let $G\subseteq \mathbb{R}^d$, $d\ge 2$, be a bounded Lipschitz domain and let $ p \in \mathcal{P}^{\log}(G)$.  Then, there holds the embedding 
	$ W^{1,p(\cdot)}(G)\hookrightarrow\hookrightarrow L^{p(\cdot)}(G)$.
\end{Sa}
\begin{Bew}
	See \cite[Thm. 3.8 (iv)]{KR91}, \cite[Thm. 8.4.5]{lpx-book}.
\end{Bew}

\begin{Sa}\label{kornp(x)}
	Let $G\subseteq \mathbb{R}^d\!$, $d\in \mathbb{N}$, be a bounded
	Lipschitz domain and let $\smash{p\in\mathcal{P}^{\log}(G)}$ with
	$p^- > 1$.  Then, there exists a constant $c>0$ such that
	$\smash{ \|\nabla \bfu\|_{p(\cdot)}\leq c\,( \|\bfD \bfu\|_{p(\cdot)}+\| \bfu\|_{p(\cdot)})}$
	for~every~$\smash{\bfu\in W^{1,p(\cdot)}(G)}$.
\end{Sa}
\begin{Bew}
	See 
	 \cite[Thm.~14.3.23]{lpx-book}.
\end{Bew}

\pagebreak
\begin{Sa}\label{thm:Ltp(x)}Let $G\subseteq \mathbb{R}^d$,
	$d\in\mathbb{N}$, be a bounded
	Lipschitz~domain,~${p\in\mathcal{P}^{\log}(G)}$ with $p^->1$ and
	 $(\bfu^n)_{n\in \mathbb{N}} \subseteq \smash{W^{1,p(\cdot)}_0(G)}$  such that
	$\bfu^n \weakto \bfzero$ in $\smash{W^{1,p(\cdot)}_0(G)}$ $(n\to \infty)$.  Then,
	for any ${j, n\hspace*{-0.1em}\in\hspace*{-0.1em} \setN}$, there
	exist
	$\bfu^{n,j}\hspace*{-0.2em} \in\hspace*{-0.1em} W^{1,\infty}_0(G)$
	and~$\smash{\lambda_{n,j}\hspace*{-0.1em} \in \hspace*{-0.1em}\big
		[2^{2^j}\hspace*{-0.1em}, 2^{2^{j+1}}\big ]} $~such~that
	\begin{align}
		\begin{split}
			\smash{ \lim_{n\to \infty}} \big ( {\sup}_{j \in \setN}
			\norm{\bfu^{n,j}}_{\infty}\big ) &=0\,,\\
			\norm{\nabla \bfu^{n,j}}_{\infty} &\leq c\, \lambda_{n,j}\,,
			\\
			\bignorm{\nabla \bfu^{n,j}\, \chi_{
					\set{\bfu^{n,j} \not= \bfu^n}}}_{p(\cdot)} &\leq c\, \big\| \lambda_{n,j} \chi_{\set{\bfu^{n,j} \not= \bfu^n}}\big\|_{p(\cdot)} \,,
			\\[-1mm]
			\smash{\limsup _{n \to \infty} }\,\big\| \lambda_{n,j} \chi_{\set{\bfu^{n,j} \not= \bfu^n}}\big\|_{p(\cdot)} &\leq c\, 2^{-j/p^+}\,,
		\end{split}\label{eq:C18px}
	\end{align}
	where $c=c(d,p,G)>0$.  Moreover, for any $j \in
	\setN$, $\nabla \bfu^{n,j} \weakto \bfzero $~in~$L^s(G) $~${(n \to
		\infty)}$,
	$s \in [1,\infty)$, and $\nabla \bfu^{n,j} \stackrel{*}{\weakto}
	\bfzero$ in $L^\infty(G)$ $(n \to
	\infty)$. 
\end{Sa}
\begin{Bew}
	See \cite[Thm.~4.4]{dms}, \cite[Cor.~9.5.2]{lpx-book}.
\end{Bew}

\section{The electric field $\bE$}\label{sec:E}

We first note that the system \eqref{maxwell}  is
separated from \eqref{NS}, in the sense that one can first solve the quasi-static
Maxwell's equations yielding an electric~field~$\bE$, which then, in turn, enters into \eqref{NS} as a parameter through the stress tensors.

It is proved in \cite{pi81}, \cite{pi84}, \cite{rubo}, that for bounded Lipschitz domains,~there
exists a solution\footnote{Here, we employ the standard function spaces
  $H(\curl):=\set{\bv \in L^2(\Omega)\fdg \curl \bv \in L^2(\Omega)}$,
  $H(\divo):=\set{\bv \in L^2(\Omega)\fdg \divo \bv \in L^2(\Omega)}$
  and $H^{-1/2}(\partial \Omega):= (H^{1/2}(\partial \Omega))^*$.\vspace*{-0.5cm}}
$\E \in H(\curl) \cap H(\divo)$ of the system (\ref{maxwell})
with ${\norm{\bE}_2 \le c\, \norm{\bE_0}_{H^{-1/2}(\partial \Omega)}}$.
A more detailed analysis of the properties of the electric field $\bE$
can be found~in~\cite{frank-phd}. Let us summarize these results here.
Combining $\eqref{maxwell}_1$ and $\eqref{maxwell}_2$,~we~obtain~that
\begin{align}
  \smash{-\Delta\mathbf{E}=\curl\curl \mathbf{E}-\nabla \Div\mathbf{E}=0\,,}\label{eq:harmonic}
\end{align}
i.e.,~the electric field is a harmonic function and, thus, real analytic.~In~particular, for a harmonic function, we can characterize its zero set as follows:

\begin{Lem}\label{Untermannigfaltigkeit}
  Let $\Omega\subseteq\R^d$, $d\in \mathbb{N}$, be a bounded domain and
  $u:\Omega\to\R$~a~non-trivial analytic function. Then, 
 $u^{-1}(0)$ is a union of $C^1$--manifolds~$(M_i)_{i=1,\cdots,m}$, $m\in \mathbb{N}$,
  with $\dim M_i\leq d-1$ for every $i=1,\cdots,m$, and $\abs{u^{-1}(0)}=0$.
\end{Lem}

\begin{Bew} See \cite{frank-phd}, \cite[Lem.~3.1]{erw}.
\end{Bew}

Finally, we observe that using the regularity theory for Maxwell's equations 
(cf.~\cite{gsch}, \cite{rubo}), one can give conditions on the
boundary data $\bE_0$ ensuring that the electric field $\bE$ is globally
bounded, i.e., $\|\bE\|_\infty \le c(\bE_0)$. Based on these two
observations, we will make the following assumption on the electric
field $\bE$: 
\begin{Vss}\label{VssE}
  The electric field $\mathbf{E}$ satisfies $\mathbf{E}\in C^\infty(\Omega) \cap
  L^\infty(\Omega)$, and the closed set $\abs{\bE}^{-1}(0)$ is a null
  set, i.e., 
   ${\Omega_0:= \{ x\in \Omega \fdg \vert
  \mathbf{E}(x)\vert>0\}}$~has~full~measure.
\end{Vss}

Note that there, indeed, exist solutions of the quasi-static Maxwell's equations that satisfy
Assumption \ref{VssE}, but do not belong to any H\"older space. In
particular, there exist  a solution 
of the quasi-static Maxwell's equations such that for a standard choice of $\hat p\in \mathcal{P}^{\log}(\mathbb{R})$, we have that
$p:=\hat p\circ \vert \bE\vert^2\notin \mathcal{P}^{\log}(\Omega)$.

\begin{Bem}
	Let $\Omega:=[-2,0]\times [-1,1]\subseteq \mathbb{R}^2$ and
        let $\bE_0\in H(\textup{div};\Omega)$  be a vector field
        defined via $\bE_0(x_1,x_2):=(1/\log(\vert
        \log(\frac{1}{4}\vert x_2\vert )\vert),0)^\top $ for every
        $(x_1,x_2)^\top\in \Omega$. Then,  in analogy with
        \cite[Thm. 3.21]{rubo}, a solution $\bE\in
        H(\textup{curl};\Omega)\cap H(\textup{div};\Omega)$ of the
        quasi-static Maxwell's equation with prescribed
        data~${\bE_0\in H(\textup{div};\Omega)}$~is~given via the gradient of a solution $u\in W^{1,2}(\Omega)/\mathbb{R}$ of the Neumann problem 
	\begin{align}
          \begin{aligned}
            -\Delta u&= 0&&\quad\text{ in }\Omega\,,\\
            \nabla u\cdot\bn&=\bE_0\cdot\bn&&\quad\text{ on
            }\partial\Omega\,,
          \end{aligned}\label{eq:Neumann}
	\end{align}
	i.e., $\bE=\nabla u$. With the help of an approximation of \eqref{eq:Neumann} using finite elements, the following pictures for the electric field are obtained:
	\begin{figure}[h!]
			\centering
			\hspace*{-0.75cm}\includegraphics[width=13cm]{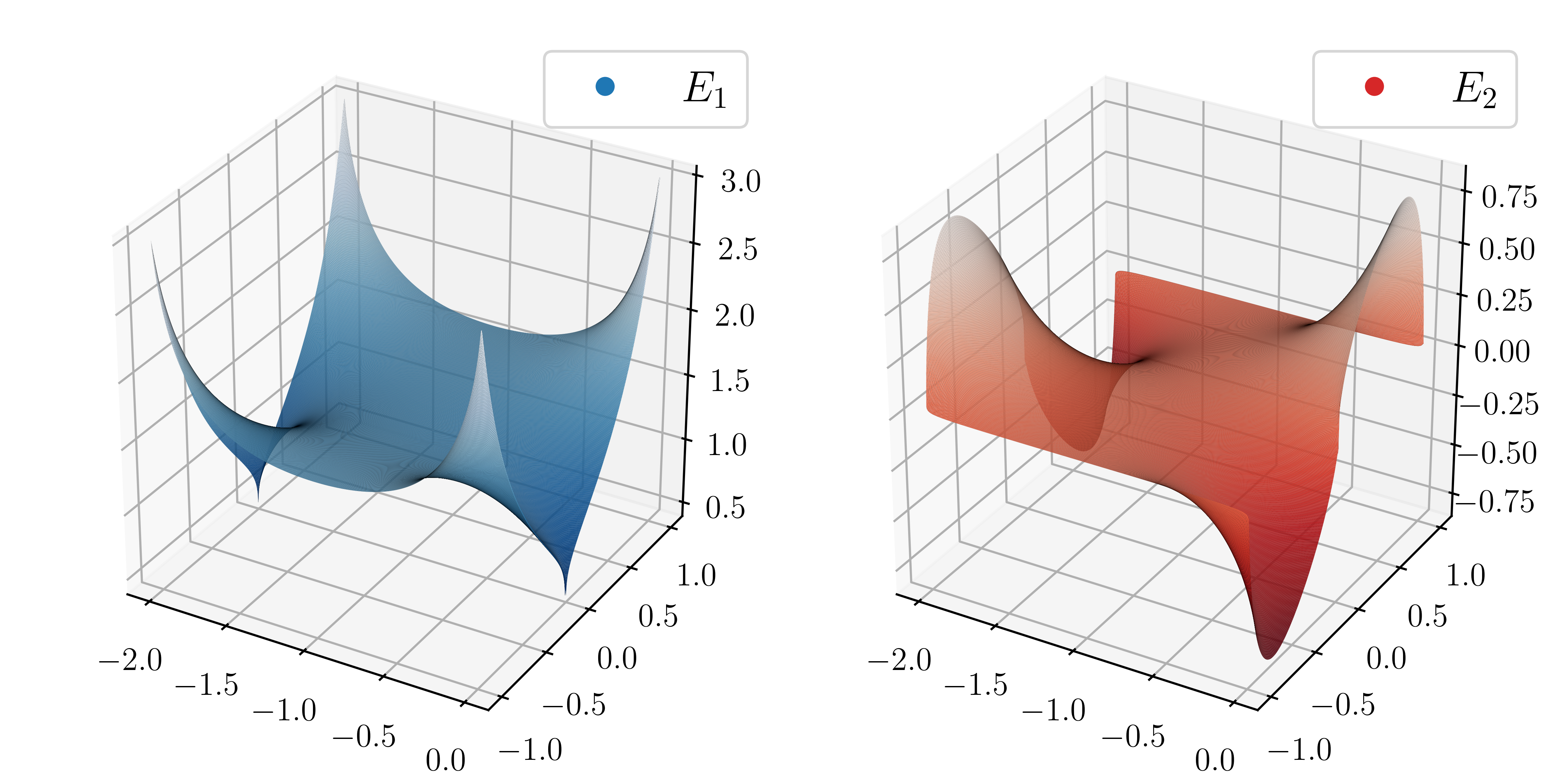}
			\caption{Plots of the first (blue/left) and the second (red/right) component of the numerically determined electric field.}
	\end{figure}

	\hspace*{-6mm}These pictures indicate that $\bE\in C^0(\overline{\Omega})$, or at least that $\bE\in L^\infty(\Omega)$. In addition, \eqref{eq:harmonic} in conjunction with Weyl's lemma imply that  $\bE\in C^\infty(\Omega)$. Note that since $E_1:=\bE\cdot \be_1=\bE_0\cdot \be_1$ on $\{0\}\times [-1,1]$, we find that ${\bE\notin C^{0,\alpha}(\overline{\Omega})}$~for~any~${\alpha\in \left(0,1\right]}$. Apart from that, if $\hat p\in \mathcal{P}^{\log}(\mathbb{R})$ is given via $\hat p(x):=1/\log(e+1/\vert x\vert)$ for~all~${x\in \mathbb{R}}$, then it is easily checked that $p:=\hat p\circ \vert \bE\vert^2$ satisfies ${p(0,x_2)\log(e+1/\vert x_2\vert)\to \infty}$ as $x_2\to  0$ and, thus, $p\notin \mathcal{P}^{\log}(\Omega)$.
\end{Bem}

In the sequel, we do not use that $\bE$ is the solution of the
quasi-static Maxwell's equations~\eqref{maxwell}, but we will only use
Assumption \ref{VssE}.  

\begin{Sa}\label{compactnew}
  Let $\Omega\subseteq \R^d$,
  ${d\in \mathbb{N}}$,~be~open,~$p\in\left[1,\infty\right)$ and let
  Assumption~\ref{VssE} be satisfied. Set $\smash{p^*:=\frac{dp}{d-p}}$
  if $p<d$ and $p^*:=\infty$ if $p\ge d$. Then, for any open~set
  $\Omega'\subset\subset \Omega$~with~${\partial\Omega'\in C^{0,1}}$
  and   any $\smash{\alpha\ge 1+\frac{2}{p}}$, it holds
  \begin{align*}
    \smash{X^{p,p}_{\nabla}(\Omega;\vert \mathbf{E}\vert^2) \vnor
    L^{r}(\Omega';\vert \bfE\vert^{\alpha r})}
  \end{align*}
  with $r \in [1,p^*]$ if $p\neq d$ and $r \in [1,p^*)$ if $p=d$.
\end{Sa}

\begin{proof}
	See \cite[Thm.~3.3]{kr-micro}.
\end{proof}

\section{A weak stability lemma}\label{sec:stab}
The weak stability of problems of $p$--Laplace type is well-known
(cf.~\cite{dms}).~It~also holds for our problem \eqref{NS} if we make
appropriate natural assumptions on the extra stress tensor $\mathbf{S}$ and on the couple stress tensor
$\mathbf{N}$, which are motivated by the canonical example
in~\eqref{eq:SN-ex}.~We~denote the symmetric and the skew-symmetric part, resp., of a
tensor $\mathbf{A}\hspace*{-0.1em}\in\hspace*{-0.1em} \mathbb{R}^{d\times d}$ by ${\mathbf{A}^{\sym}\hspace*{-0.1em}:=\hspace*{-0.1em}\frac{1}{2}(\mathbf{A}\hspace*{-0.1em}+\hspace*{-0.1em}\mathbf{A}^\top)}$~and~${\mathbf{A}^{\anti}\hspace*{-0.1em}:=\hspace*{-0.1em}\frac{1}{2}(\mathbf{A}
	\hspace*{-0.1em}-\hspace*{-0.1em}\mathbf{A}^\top)}$. Moreover,~${\R^{d\times
		d}_{\sym}\hspace*{-0.1em}:=\hspace*{-0.1em}\{\mathbf{A}\hspace*{-0.1em}\in\hspace*{-0.1em} \R^{d\times d}\hspace*{-0.1em} \mid\hspace*{-0.1em}
	\mathbf{A}\hspace*{-0.1em}=\hspace*{-0.1em}\mathbf{A}^\sym\}}$ and $\R^{d\times
	d}_{\anti}\hspace*{-0.1em}:=\hspace*{-0.1em}\{\mathbf{A}\in \R^{d\times d}\hspace*{-0.1em}\mid\hspace*{-0.1em}
{\mathbf{A}\hspace*{-0.1em}=\hspace*{-0.1em}\mathbf{A}^\anti}\}$.
\begin{Vss}\label{VssSpx}
		For the extra stress tensor
		$\mathbf{S}:\R_{\sym}^{d\times d}\times \R_{\anti}^{d\times
		          d}\times\R^d\to \setR^{d}$  and some $\hat p \in \mathcal
	        P^{\log}(\setR)$ with $\hat p^- >1$, there
	        exist constants $c,C >0$ such that: 
		\begin{enumerate}
				\item[{\rm \hypertarget{(S.1)}{(S.1)}}] $\mathbf{S}\in C^0(\R_{\sym}^{d\times d}\times \R_{\anti}^{d\times d}\times \R^d;\setR^{d
						\times d})$.
				\item[{\rm \hypertarget{(S.2)}{(S.2)}}] For every $\bD \in \R_{\sym}^{d\times d}$,
				$\bR \in \R_{\anti}^{d\times d}$ and $\bE \in \setR^d$, it holds\\[-5mm]
				\begin{align*}
								\vert\mathbf{S}^{\sym}(\mathbf{D},\mathbf{R},\mathbf{E})\vert&\leq
								c\,\big (1+\vert\E\vert^2 \big) \big (1+\vert\mathbf{D}\vert^{\hat p(\abs{\bE}^2)-1}\big)\,,
								\\ 
								\vert\mathbf{S}^{\anti}(\mathbf{D},\mathbf{R},\mathbf{E})\vert&\leq
								c\,\vert \mathbf{E}\vert^2 \big (1+\vert\mathbf{R}\vert^{\hat p(\abs{\bE}^2)-1}\big)\,.
					\end{align*}\\[-7mm]
				
				\item[{\rm \hypertarget{(S.3)}{(S.3)}}]  For every $\bD \in \R_{\sym}^{d\times d}$,
				$\bR \in \R_{\anti}^{d\times d}$ and $\bE \in \setR^d$, it holds\\[-5mm]
				\begin{align*}
								\mathbf{S}(\mathbf{D},\mathbf{R},\mathbf{E}):\mathbf{D}
								&\geq  c\,\big (1+\vert\E\vert^2\big )\, \big( \vert\mathbf{D}\vert^{\hat p(\abs{\bE}^2)}-C\big)\,,
								\\
								\mathbf{S}(\mathbf{D},\mathbf{R},\mathbf{E}):\mathbf{R}&\geq
								c\,\vert \mathbf{E}\vert^2 \big( \vert\mathbf{R}\vert^{\hat p(\abs{\bE}^2)}-C\big)\,.
					\end{align*}\\[-7mm]
				
				\item[{\rm \hypertarget{(S.4)}{(S.4)}}]  For every $\bD_1, \bD_2 \in
				\R_{\sym}^{d\times d}$, $\bR_1, \bR_2 \in \R_{\anti}^{d\times d}$
				and $\bE \in \setR^d$ with $(\mathbf{D}_1,\vert
				\mathbf{E}\vert \mathbf{R}_1)\neq(\mathbf{D}_2,\vert
				\mathbf{E}\vert \mathbf{R}_2)$, it holds\\[-6mm]
				\begin{align*}
								\big (\mathbf{S}(&\mathbf{D}_1,\mathbf{R}_1,\mathbf{E})-
								\mathbf{S}(\mathbf{D}_2,\mathbf{R}_2,\mathbf{E})\big ):
								\big (\mathbf{D}_1-\mathbf{D}_2+\mathbf{R}_1-\mathbf{R}_2\big )>0\,.
					\end{align*}
			\end{enumerate}
	\end{Vss}
\begin{Vss}\label{VssNpx}
		For the couple stress tensor
		$\mathbf{N}:\R^{d\times d}\times \R^d\to\R^{d\times d}$ and some $\hat p \in \mathcal
	        P^{\log}(\setR)$ with $ \hat p^- >1$, there
	        exist constants $c, C >0$ such that:  
		\begin{enumerate}
				\item[{\rm \hypertarget{(N.1)}{(N.1)}}]$\mathbf{N}\in C^0(\R^{d\times d}\times \R^d;\R^{d\times d})$.
				
				\item[{\rm \hypertarget{(N.2)}{(N.2)}}]  For every $\bL \in \R^{d\times d}$ and $\bE \in
				\setR^d$, it holds\\[-6mm]
				\begin{align*}
						\vert\mathbf{N}(\mathbf{L},\mathbf{E})\vert\leq c\,\big
						\vert\mathbf{E}\vert^2
						\big(1+\vert\mathbf{L}\vert^{\hat p(\abs{\bE}^2)-1}\big )\,.
					\end{align*}\\[-10mm]
				
				\item[{\rm \hypertarget{(N.3)}{(N.3)}}] For every $\bL \in \R^{d\times d}$ and $\bE \in
				\setR^d$, it holds\\[-6mm]
				\begin{align*}
						\mathbf{N}(\mathbf{L},\mathbf{E}):\mathbf{L}\geq
						c\,\big \vert\mathbf{E}\vert^2
						\big(\vert\mathbf{L}\vert^{\hat p(\abs{\bE}^2)} -C\big)\,.
					\end{align*}\\[-10mm]
		
		              \item[{\rm \hypertarget{(N.4)}{(N.4)}}]  For every $\bL_1,\bL_2 \in \R^{d\times d}$ and $\bE \in
				\setR^d$ with $\vert \mathbf{E}\vert>0$ and
				$\mathbf{L}_1\neq \mathbf{L}_2$, it holds\\[-6mm]
				\begin{align*}
						(\mathbf{N}(\mathbf{L}_1,\mathbf{E})-\mathbf{N}
						(\mathbf{L}_2,\mathbf{E})):(\mathbf{L}_1-\mathbf{L}_2)>0\,.
					\end{align*}
			\end{enumerate}
	\end{Vss}

\begin{Bem}\label{locallog}
  Let Assumption~\ref{VssE}, Assumption~\ref{VssSpx} and
  Assumption~\ref{VssNpx}
  be satisfied. Since
  $\vert\bfE\vert^2\in W^{1,\infty}(\Omega')$ for each
  $\Omega'\subset\subset \Omega$, which follows from
  ${\bfE\in C^\infty(\Omega)}$, we have that the variable exponent 
  $$
  p:=\hat p\circ \vert\bfE\vert^2 
  $$
  satisfies $p \in \mathcal{P}^{\infty}(\Omega)\cap C^0(\Omega)$ and $p|_{\Omega'}\in \mathcal{P}^{\log}(\Omega')$ for each
  $\Omega'\subset\subset \Omega$.
\end{Bem}

\begin{Lem}\label{symgradp(x)}
	Let Assumption~\ref{VssE}, 
	Assumption~\ref{VssSpx} and Assumption~\ref{VssNpx} be satisfied with ${p^-\ge \frac{2d}{d+2}}$.  Then, we have that 
	$\smash{\XDpo\hookrightarrow W^{1,p(\cdot)}_{\loc}(\Omega)}$.
\end{Lem}

\begin{proof}
	Let $\Omega'\subset\subset \Omega$ be arbitrary. Without loss
        of generality, we may assume that $\partial \Omega'\in
        C^{0,1}$. Otherwise, we switch to some $\Omega''$ such that
        ${\Omega'\subset\subset\Omega''\subset\subset \Omega}$ with ${\partial \Omega''\in C^{0,1}}$.
	 Since $\smash{p\hspace*{-0.1em}\in\hspace*{-0.1em} C^0(\Omega)}$ (cf. Remark \ref{locallog}), Theorem~\ref{poincarep(x)} implies~that~any ${\bfu \in \smash{\XDpo}}$ satisfies $\bfu\in L^{p(\cdot)}(\Omega')$ with
	\begin{align}
		\|\bfu\|_{L^{p(\cdot)}(\Omega')}\leq c(p,\Omega')\|\bfu\|_{\smash{\XDpo}}\label{symgradp(x).1}
	\end{align}
	for some constant $c(p,\Omega')\!>\!0$. Next, let
        $(\bfu_n)_{n\in \mathbb{N}}\!\subseteq
        \! \smash{\mathcal{V}^{p^-,p(\cdot)}_{\bD}}$~be~a~sequence~such~that
        $\bfu_n\to \bfu $ in $\smash{\XDpo}$ $(n\to\infty)$. Then,
        \eqref{symgradp(x).1} gives us that $\bfu_n\to \bfu $ in
        $L^{p(\cdot)}(\Omega')$ $(n\to\infty)$. In addition, since
        $\smash{p|_{\Omega'}\in \mathcal{P}^{\log}(\Omega')}$
        (cf.~Remark \ref{locallog}), Korn's inequality (cf.~Theorem
        \ref{kornp(x)} and
        $\smash{\XDpo}\vnor W^{1,p(\cdot)}(\Omega')$)  and
        \eqref{symgradp(x).1} yield 
	\begin{align}
             \|\nabla\bfu_n\|_{L^{p(\cdot)}(\Omega')}+\|\bfu_n\|_{L^{p(\cdot)}(\Omega')} &\leq
            c(p,\Omega')\big(\|\bfD\bfu_n\|_{L^{p(\cdot)}(\Omega')}
            +\|\bfu_n\|_{L^{p(\cdot)}(\Omega')}\big) \notag 
            \\
            &\le c(p,\Omega')\|\bfu_n\|_{\smash{\XDpo}}\,.
              \label{symgradp(x).2} 
	\end{align}
	Thus, 
        we observe that $(\bfu_n)_{n\in \mathbb{N}}\subseteq \smash{\mathcal{V}^{p^-,p(\cdot)}_{\bD}}$ is bounded in $W^{1,p(\cdot)}(\Omega')$. This 
	implies the existence of a vector field $\tilde\bfu\in W^{1,p(\cdot)}(\Omega')$ and of a not relabeled subsequence such that $\bfu_n\rightharpoonup \tilde\bfu$ in $W^{1,p(\cdot)}(\Omega')$ $(n\to \infty)$. Due to the uniqueness of~weak~limits, we conclude that $\bfu=\tilde\bfu$ in $W^{1,p(\cdot)}(\Omega')$ and by taking the limit inferior in \eqref{symgradp(x).2} that $\|\bfu\|_{W^{1,p(\cdot)}(\Omega')}\leq c(p,\Omega')\|\bfu\|_{\smash{\XDpo}}$.
\end{proof}

\begin{Lem}\label{hatgradp(x)}
	Let Assumption~\ref{VssE}, 
	Assumption~\ref{VssSpx} and Assumption~\ref{VssNpx} be satisfied with ${p^-\ge\frac{2d}{d+2}}$. Then, we have that
	$\smash{\XGpEo\hookrightarrow W^{1,p(\cdot)}_{\loc}(\Omega_0)}$\footnote{Recall that $\Omega_0:=\{x\in \Omega\fdg \vert \bfE(x)\vert>0\}$ (cf. Assumption \ref{VssE}).}
	and
	$\smash{\nabla (\bw|_{\Omega'})=(\hat\nabla \bw)|_{\Omega'}}$ for all ${\bw\in \smash{\XGpEo}}$ and $\Omega'\subset\subset\Omega_0$. 
\end{Lem}

\begin{proof}
  Let $\Omega'\subset\subset\Omega_0$ be arbitrary and fix some
  $\Omega''\subset\subset\Omega_0$ such that
  $\Omega'\subset\subset\Omega''$.
  Due~to~${\vert \bE\vert >0 }$~in~$\overline{\Omega''}$ and
  ${\vert \bE\vert\in C^0(\overline{\Omega''})}$, there is a constant
  $c(\Omega'')>0$ such that
  $c(\Omega'')\leq \vert \bE\vert^2$~in~$\overline{\Omega''}$. Thus,
  for every $\bw \in C_0^\infty(\Omega)$, H\"older's inequality in
  variable Lebesgue spaces and \eqref{eq:mo-no} imply  
	\begin{align}
		\|\bw\|_{{X^{p^-,p(\cdot)}_{{\nabla}}(\Omega'')}}\leq
          2\,\smash{c(\Omega'')^{\frac{-1}{p^-}}
			\|\bw\|_{{\XGpEo}}\,.}\label{hatgradp(x).1} 
	\end{align}
	Furthermore,
	since $p|_{\Omega''}\in C^0(\Omega'')$ (cf.~Remark \ref{locallog}) and $p^-\ge\frac{2d}{d+2}$, Theorem~\ref{poincarep(x)} implies that every $\bw \in C_0^\infty(\Omega)$ satisfies
	\begin{align}
		\|\bw\|_{L^{p(\cdot)}(\Omega')}\leq  c(p,\Omega',\Omega'')\|\bw\|_{{X^{p^-,p(\cdot)}_{{\bD}}(\Omega'')}}\leq  c(p,\Omega',\Omega'')\|\bw\|_{{X^{p^-,p(\cdot)}_{{\nabla}}(\Omega'')}}\label{hatgradp(x).2}
	\end{align}
	for some constant $c(p,\Omega',\Omega'')>0$. Combining \eqref{hatgradp(x).1} and \eqref{hatgradp(x).2}, we find that for every $\bw \in C_0^\infty(\Omega)$, it holds
	\begin{align}
		\|\bw\|_{W^{1,p(\cdot)}(\Omega')}\leq c(p,\Omega',\Omega'')\|\bw\|_{{\XGpEo}}\label{hatgradp(x).3}
	\end{align}
	for some constant $c(p,\Omega',\Omega'')\!>\!0$. Since 
	$\smash{\XGpEo}$ is  the closure of~$\smash{ C_0^\infty(\Omega)}$, and $C^{\infty}(\overline
        {\Omega'})$ is dense in $W^{1,p(\cdot)}(\Omega')$
        (cf.~\cite[Thm.~9.1.7]{lpx-book}) since $p|_{\Omega''}\in
        \mathcal P^{\log}(\Omega'')$ (cf.~Remark \ref{locallog}),~\eqref{hatgradp(x).3} implies 
	$\smash{\XGpEo\hspace*{-0.2em}\hookrightarrow\hspace*{-0.2em} W^{1,p(\cdot)}(\Omega')}$
	and $\smash{\nabla}(\bw|_{\Omega'})\hspace*{-0.2em}=\hspace*{-0.2em}(\hat{\nabla} \bw)|_{\Omega'}$ for every
	${\bw\hspace*{-0.2em}\in \hspace*{-0.2em}\smash{\XGpEo}}$.
\end{proof}

Now we can formulate  the following weak stability property for problem~\eqref{NS}.

\begin{Lem}\label{DalMaso3p(x)}
  Let $\Omega\subseteq\R^d$, $d\ge 2$, be a bounded domain and let
  \mbox{Assumption~\ref{VssE}}, Assumption~\ref{VssSpx} and
  Assumption~\ref{VssNpx} be satisfied with
  ${p^->\frac{2d}{d+2}}$. Moreover, let~$\smash{(\bv^n)_{n\in \mathbb{N}}\subseteq \VDpo}$ and
  $\smash{(\w^n)_{n\in \mathbb{N}}\subseteq \XGpEo}$ be~such~that \linebreak
  $(\bfR(\bfv^n,\w^n))_{n\in \mathbb{N}}\subseteq L^{p(\cdot)}(\Omega;\vert \bfE\vert^2)$ is bounded and
  \begin{align}
    \begin{aligned}\label{stab-konv}
      \bv^n&\rightharpoonup\bv\quad& &\text{in}\
      \VDpo &&\quad(n\to \infty)\,,
      \\
      \w^n&\rightharpoonup\w\quad& &\text{in}\ \XGpEo&&\quad(n\to \infty) \,.
    \end{aligned}
  \end{align}      
  For every ball ${B\subset\subset \Omega_0}$ such that
  ${B':=2B\subset \subset \Omega_0}$ and $\tau\in C_0^\infty(B')$
  satisfying ${\chi_B\leq \tau\le\chi_{B'}}$, we set
  ${ \mathbf{u}^n:=(\mathbf{v}^n-\mathbf{v})\tau}$,
  ${\bfpsi^n:=(\boldsymbol\omega^n-\boldsymbol\omega)\tau} \in
  W^{1,p(\cdot)}_0(B')$, $n\in \mathbb{N}$.  Let
  $ \mathbf{u}^{n,j}\in W^{1,\infty}_0(B')$, $n,j\in \mathbb{N}$, and
  $\bfpsi^{n,j}\in W^{1,\infty}_0(B')$, $n,j\in \mathbb{N}$, resp.,
  denote the Lipschitz truncations constructed according to
  Theorem~\ref{thm:Ltp(x)}.~Furthermore, assume that for every
  ${j\in \mathbb{N}}$, we have that
	\begin{align} 
		\limsup_{n\to\infty}\big\vert\big\langle&\Sn-\Ss,
		\D\bu^{n,j}+\Rr(\bu^{n,j},\bfpsi^{n,j})\big\rangle \notag
		\\[-1mm]
		&\quad + \big\langle\Nn-\N,\nabla\bfpsi^{n,j}\big\rangle
		\big \vert \le
		\delta_j\,, \label{mon4}
	\end{align}
	where $\delta_j\to 0$ $(j\to 0)$. \!Then, \!one has $
	\nabla\bv^n\to\nabla\bv$ a.e.~in $B$~${(n\to \infty)}$,~${\nabla\w^n\to\nabla\w}$ a.e.~in $B$ $(n\to \infty)$ and $\w^n\to\w$ a.e.~in $B$ $(n\to \infty)$~for~a~suitable~subsequence.
\end{Lem}

\begin{Bem}\label{hatgrad}
(i) For each ball $B'\subset\subset\Omega_0$, Lemma
    \ref{symgradp(x)} and Lemma \ref{hatgradp(x)} yield the embeddings
    $\smash{\XDpo},\smash{\XGpEo}\hookrightarrow\smash{W^{1,p(\cdot)}(B')}$
    and,~therefore,~that ${\bfv^n-\bv,\w^n-\w\!\in\! W^{1,p(\cdot)}(B')}$ for
    all $n\!\in\! \mathbb{N}$, which is crucial for the applicability of
    the Lipschitz truncation technique
    (cf.~Theorem~\ref{thm:Ltp(x)}). In particular,
    Lemma~\ref{hatgradp(x)} yields that
    ${\nabla(\w|_{B'})=(\hat{\nabla}\w)|_{B'}}$ in $L^{p(\cdot)}(B')$
    for all $\w \in \smash{\XGpEo}$, which is precisely the sense in
    which the gradients of both $\w\in \smash{\XGpEo}$ and
    ${(\w^n)_{n\in \mathbb{N}}\in \smash{\XGpEo}}$ are to be
    understood~in~\eqref{mon4}.\\[-3mm]

(ii) The boundedness of
    $(\bfR(\bfv^n,\w^n))_{n\in \mathbb{N}}\subseteq
    L^{p(\cdot)}(\Omega;\vert \bfE\vert^2)$ yields a tensor field
    $\widehat{\bR}\in L^{p(\cdot)}(\Omega;\vert \bfE\vert^2)$ such
    that up to a not relabeled subsequence
    \begin{align}
      \bfR(\bfv^n,\w^n)\rightharpoonup\widehat{\bR}\quad\text{in}\
      L^{p(\cdot)}(\Omega;\vert \bfE\vert^2)\quad(n\to \infty)
      \,.\label{hatgrad1} 
    \end{align}
    Using the embedding $L^{p(\cdot)}(\Omega)\vnor L^{p^-}(\Omega)$,
    Korn's inequality for the constant exponent $p^-\ge\frac{2d}{d+2}$
    and $L^{p^-}(\Omega)\vnor L^{p^-}(\Omega;\abs{\bE}^2)$, since
    ${\bfE\in L^\infty(\Omega)}$, we deduce from \eqref{stab-konv}
    that
    \begin{align}
      \begin{aligned}
        \nabla\bv^n&\rightharpoonup\nabla\bv\quad& &\text{in}\
        L^{p^-}(\Omega;\vert\bfE\vert^2)&&\quad(n\to \infty)\,,
        \\
        \w^n\cdot\boldsymbol\epsilon&\rightharpoonup\w\cdot\boldsymbol\epsilon\quad&
        &\text{in}\ L^{p^-}(\Omega;\vert\bfE\vert^2) &&\quad(n\to
        \infty)\,.
      \end{aligned}\label{hatgrad2}
    \end{align}
    Combining \eqref{hatgrad1}, \eqref{hatgrad2}, the embedding
    $L^{p(\cdot)}(\Omega;\abs{\bE}^2)\vnor
    L^{p^-}(\Omega;\abs{\bE}^2)$ and the definition of
    $\bfR(\bfv,\w)$, we conclude that
    $\bfR(\bfv,\w)=\widehat{\bfR}\in
    L^{p(\cdot)}(\Omega;\vert\bfE\vert^2)$ in \eqref{hatgrad1}. Thus,
    the expression with $\bfR(\bfv,\w)$ in \eqref{mon4} is well-defined.
\end{Bem}

\begin{proof}[Proof of Lemma~\ref{DalMaso3p(x)}]
  In view of the embeddings
  $\smash{\VDpo\hookrightarrow W^{1,p(\cdot)}(B')}$, $\smash{\XGpEo\hookrightarrow W^{1,p(\cdot)}(B')}$ 
  (cf.~Lemma~\ref{symgradp(x)} and Lemma~\ref{hatgradp(x)}) as well as
  \linebreak 
  $W^{1,p(\cdot)}(B')\vnor W^{1,p^-}(B')$, we~deduce~from \eqref{stab-konv}, also using Rellich's~compact-ness~theorem for
  constant exponents and Theorem \ref{rellichp(x)}, that
  \begin{align}
    \begin{aligned}\label{konvergenz-stab}
      \bv^n&\rightharpoonup\bv\quad& &\text{in}\
      W^{1,p(\cdot)}(B')&&\quad(n\to \infty)\,, \\\bv^n&\to\bv\quad&
      &\text{in}\ L^q(B')\cap L^{p(\cdot)}(B')\text{ and a.e. in }
      B' &&\quad(n\to \infty)\,,\\
      \w^n&\rightharpoonup\w\quad& &\text{in}\
      W^{1,p(\cdot)}(B')&&\quad(n\to \infty)\,,
      \\
      \w^n&\to\w\quad& &\text{in}\ L^q(B')\cap L^{p(\cdot)}(B') \text{
        and a.e. in } B'&&\quad(n\to \infty)\,,
    \end{aligned}
  \end{align}
  where $q\hspace*{-0.1em}\in \hspace*{-0.1em}\left[1,(p^-)^*\right)$.
  \!Throughout the proof, we will employ the~particular~notation
  \begin{align}
    \begin{aligned}\label{eq:not}
      \widetilde\bS&:=\Ss\,, \quad &&\mathbf{S}^n\,:=\Sn\,,
      \\
      \widetilde\bN&:=\N\,, \quad &&\mathbf{N}^n:=\Nn\,.
    \end{aligned}
  \end{align}
  Using (\hyperlink{(S.2)}{S.2}), (\hyperlink{(N.2)}{N.2}), Assumption
  \ref{VssE}, \eqref{stab-konv} and Remark \ref{hatgrad}, we obtain a constant
  ${K:=K(\bE)>0}$ (not depending on $n\in \mathbb{N}$) such
  that 
\begin{equation}
  \label{stab-est}
  \begin{aligned}
\hspace*{-3mm}     \Vert\bv^n\|_{ \smash{\VDpo}}\! +\|\bv \Vert_{\smash{\VDpo}}\!+ 
     \Vert\w^n\|_{p^-,
      p(\cdot),\abs{\bE}^2}\!+\|\w\Vert_{p^-, p(\cdot),\abs{\bE}^2}&\le
    K\,, \hspace*{-7.5mm}
    \\
    \|\bfR(\bv^n,\w^n)\|_{ p(\cdot),\abs{\bE}^2}+\|\bfR(\bv,\w)\|_{ p(\cdot),\abs{\bE}^2}&\le K\,,
    \\
    \hspace*{-2mm}\Vert\mathbf{S}^n\Vert_{p'(\cdot)} \! +  \Vert\widetilde{\mathbf{S}}\Vert_{p'(\cdot)}\!+  \Vert
    (\mathbf{S}^n)^\anti\Vert_{\smash{p'(\cdot),\vert
      \bfE\vert^{\smash{\frac{-2}{p(\cdot)-1}}}}}\!+\Vert
    \widetilde{\mathbf{S}}^\anti\Vert_{\smash{p'(\cdot),\vert
      \bfE\vert^{\smash{\frac{-2}{p(\cdot)-1}}}}} &\le K\,,     \hspace*{-7.5mm}
    \\
    \Vert\mathbf{N}^n\Vert_{\smash{p'(\cdot),\vert
      \bfE\vert^{\smash{\frac{-2}{p(\cdot)-1}}}}}+\Vert\widetilde {\mathbf{N}}
    \Vert_{\smash{p'(\cdot),\vert \bfE\vert^{\smash{\frac{-2}{p(\cdot)-1}}}}}&\le
    K\,.
  \end{aligned}
\end{equation}
	Recall that $\tau\in C_0^\infty(B')$ with
	$\chi_B \le \tau\leq \chi_{B'}$. Hence, using (\hyperlink{(S.4)}{S.4}) and~(\hyperlink{(N.4)}{N.4}),~we~get
	\begin{align}
		&I^n\!:=\!\int_B\! \big [\!\big (\bS^n\!-\!\widetilde \bS\big )\!:\! \big
		(\D(\bv^n\!-\!\bv)\!+\!\Rr(\bv^n\!-\!\bv,\w^n\!-\!\w) \big)
		\!+\!\big(\bN^n\!-\!\widetilde \bN\big )\!:\!\nabla(\w^n\!-\! \w)\big)^\theta\, dx \notag
		\\
		&\le\!\int_{B'}\!\big [\big (\bS^n\!-\!\widetilde \bS\big )\!:\!
		\big (\D(\bv^n\!-\!\bv)\!+\!\Rr(\bv^n\!-\!\bv,\w^n\!-\!\w) \big)
		\!+\!\big(\bN^n\!-\!\widetilde \bN\big )\!:\!\nabla(\w^n\!-\! \w) \big)^\theta \!\tau ^\theta \,dx \notag
		\\
		&\le\int_{B'}\big [\big (\bS^n-\widetilde \bS\big ):
		\big (\D(\bv^n-\bv)+\Rr(\bv^n-\bv,\w^n-\w) \big)
		\tau \big)^\theta \, dx \label{eq:I}
		\\
		&\quad +\int_{B'}\big [\big(\bN^n-\widetilde \bN\big
		):\nabla(\w^n- \w) \tau \big)^\theta 
		dx =:\int_{B'} \alpha_n^\theta  \,dx +\int_{B'}
		\beta_n^\theta \, dx \,,\notag 
	\end{align}
	where we also used that
        \begin{gather}
          \smash{\frac{1}{2} }(a^\theta+ b^\theta ) \le (a+ b )^\theta \le
          a^\theta+ b^\theta\label{eq:1}
        \end{gather}
        valid for all $a,b\ge 0$~and~${\theta \in (0,1)}$. Then,
	splitting the integral of $\alpha_n^\theta$ over $B'$ into an
	integral over $\{\mathbf{u}^n\neq\mathbf{u}^{n,j}\} $ and one over
	$\{\mathbf{u}^n=\mathbf{u}^{n,j}\} $, also using H\"older's inequality 
	with exponents $\frac{1}{\theta},\frac{1}{1-\theta}$, 
	we~find~that
	\begin{align}
		\int_{B'} \alpha_n^\theta \, dx  &\le \| \alpha_n
		\|_{L^1( B')}^\theta
		\vert\{\mathbf{u}^n\neq\mathbf{u}^{n,j}\}\vert^{1-\theta} + \|
		\alpha_n \chi_{\{\mathbf{u}^n=\mathbf{u}^{n,j}\} }
		\|_{L^1(B')}^\theta\vert B'\vert^{1-\theta}\notag
		\\
		&=:(I_{1}^n)^{\theta} \vert\{\mathbf{u}^n\neq\mathbf{u}^{n,j}\}\vert^{1-\theta}
		+ \|	\alpha_n \chi_{\{\mathbf{u}^n=\mathbf{u}^{n,j}\} }
           \|_{L^1(B')}^\theta\vert B'\vert^{1-\theta}
           \,.\label{eq:i1}
	\end{align}
	For the first term, we will use  \eqref{eq:C18px}$_4$ and, thus, have to
	show~that~$(I_1^n)_{n\in \mathbb{R}}\subseteq \mathbb{R}$~is bounded.
	In fact,  by combining \eqref{stab-konv},
	\eqref{stab-est},  
	$\tau \le 1$, as well as
        \begin{align}\label{eq:symm}
          \bA:(\bD\bu + \bR(\bu,\bw))
          = \bA: \nabla \bu +
          \bA^\anti:(\bfepsilon\cdot \bw )\,,
        \end{align}
	valid for vector fields $\bu$, $\bw$ and tensor~fields~$\bA$,
         we observe that
	\begin{align}
          \begin{aligned}
            I_{1}^n &\le \big (\Vert(\bS^n)^\sym\Vert_{p'(\cdot)} +
            \Vert\widetilde \bS^\sym\Vert_{p'(\cdot)} \big ) \Vert\bfD
            \bv^n-\bfD \bv \Vert_{p(\cdot)}\label{eq:i1est}
            \\
            &\quad + \big
            (\Vert(\bS^n)^\anti\Vert_{\smash{p'(\cdot),\vert
                \bfE\vert^{\smash{\frac{-2}{p(\cdot)-1}}}}} + \Vert
            \widetilde \bS^\anti\Vert_{\smash{p'(\cdot),\vert
                \bfE\vert^{\smash{\frac{-2}{p(\cdot)-1}}}}} \big )
            \times
            \\
            &\qquad\quad
            \times\Vert\bfR(\bfv^n,\w^n)-\bfR(\bfv,\w)\Vert_{
              p(\cdot),\abs{\bE}^2} 
          \le 2\,K^2\,.
          \end{aligned}
	\end{align}
	Similarly, we deduce that
	\begin{align}
		\int_{B'} \beta_n^\theta \, dx  &\le \| \beta_n
		\|_{L^1( B')}^\theta
		\vert\{\bfpsi^n\neq\bfpsi^{n,j}\}\vert^{1-\theta}\notag
		+ \|\beta_n \chi_{\{\bfpsi^n=\bfpsi^{n,j}\}  }
		\|_{L^1(B')}^\theta\vert B'\vert^{1-\theta}
		\\
		&=:(I_{2}^n)^{\theta} \vert\{\bfpsi^n\neq\bfpsi^{n,j}\}\vert^{1-\theta}
		+ \|\beta_n \chi_{\{\bfpsi^n=\bfpsi^{n,j}\}  }
           \|_{L^1(B')}^\theta\vert B'\vert^{1-\theta} 
           \,,\label{eq:i3}
	\end{align}
	and that
	\begin{align}
          \begin{aligned}
            \hspace*{-2mm} I_{2}^n &\le \big (\Vert\bN^n\Vert_{\smash{p'(\cdot),\vert \bfE\vert^{\smash{\frac{-2}{p(\cdot)-1}}}}} +
            \Vert\widetilde \bN\Vert_{\smash{p'(\cdot),\vert \bfE\vert^{\smash{\frac{-2}{p(\cdot)-1}}}}}
            \big )\Vert\nabla \w^n-\nabla \w \Vert_{p(\cdot),\abs{\bE}^2}
          \le K^2\,.\hspace*{-4mm} 
          \end{aligned}\label{eq:i3est}
	\end{align}
	Using \eqref{eq:i1}, \eqref{eq:i1est}--\eqref{eq:i3est} and \eqref{eq:1}
	we, thus, conclude that
	\begin{align}
          \begin{aligned}
            &\int_{B'} \alpha_n^\theta \, dx +\int_{B'} \beta_n^\theta \, dx
            \\
            &\le 2^\theta K^{2\theta} \big
            (\vert\{\bfu^n\neq\bfu^{n,j}\}\vert^{1-\theta}
            +\vert\{\bfpsi^n\neq\bfpsi^{n,j}\}\vert^{1-\theta} \big )
             \\
            &\quad + 2\,\vert B'\vert^{1-\theta} \, 
            \bigg(\int_{B'}{\alpha_n \chi_{\{\bfu^n=\bfu^{n,j}\}}\,dx}
            +\int_{B'}{\beta_n\,\chi_{\{\bfpsi^n=\bfpsi^{n,j}\}}\,dx}
            \bigg)^{\smash{\theta}}\,.
          \end{aligned}\label{eq:est-ab}
	\end{align}
	Let us now treat the last two integrals, which we denote by
	$I_3^{n,j}$ and $I_4^{n,j}$.  We~have that $ \nabla (\bfv^n -\bfv) \tau = \nabla\bu^{n,j}-
	(\bv^n-\bv) \otimes \nabla \tau$ on $\{\mathbf{u}^n=\mathbf{u}^{n,j}\}
	$, which, using \eqref{eq:symm}, 
	 implies that
	\begin{align}
	\begin{aligned}
		I_3^{n,j}&= \bigskp{ \bfS^n- \widetilde \bfS}{\big(\nabla\bu^{n,j}-
			(\bv^n-\bv) \otimes \nabla \tau \big)
			\chi_{\{\mathbf{u}^n=\mathbf{u}^{n,j}\}}}
		\\
		&\quad +\bigskp{ \big(\bfS^n- \widetilde \bfS\big )^\anti
		}{\bfepsilon \cdot \bfpsi^{n,j} \chi_{\{\bfpsi^n=\bfpsi^{n,j}\}
		}}
		\\
		&\quad +\bigskp{ \big(\bfS^n- \widetilde \bfS\big )^\anti
		}{\bfepsilon \cdot (\w^n -\w)\tau 
			\chi_{\{\mathbf{u}^n=\mathbf{u}^{n,j}\}\cap
				\{\bfpsi^n\neq\bfpsi^{n,j}\} }}
		\\
		&\quad -\bigskp{ \big(\bfS^n- \widetilde \bfS\big )^\anti
		}{\bfepsilon \cdot (\w^n-\w)\tau 
			\chi_{\{\mathbf{u}^n\neq\mathbf{u}^{n,j}\}\cap
				\{\bfpsi^n=\bfpsi^{n,j}\}}}\,.
				\end{aligned}\label{eq:I5}
	\end{align}
	We have $ \nabla (\w^n
	-\w) \tau \hspace*{-0.1em}=\hspace*{-0.1em} \nabla\bfpsi^{n,j}\hspace*{-0.1em}-\hspace*{-0.1em} (\w^n-\w) \otimes \nabla \tau $ on $\{\bfpsi^n\hspace*{-0.1em}=\hspace*{-0.1em}\bfpsi^{n,j}\} $,
	which~implies~that
	\begin{align}
		\begin{aligned}
		I_4^{n,j}= \bigskp{ \bfN^n- \widetilde \bfN}{\big(\nabla \bfpsi^{n,j}-(
				\w^{n}-\w)\otimes \nabla \tau 
			\big)	\chi_{\{\bfpsi^n=\bfpsi^{n,j}\} }}\,.
		\end{aligned}\label{eq:I6}
	\end{align}
	Using  \eqref{eq:symm}
	and adding  suitable terms,
	we deduce from \eqref{eq:I5} and~\eqref{eq:I6}~that
	\begin{align}\begin{aligned}
		I^{n,j}_{3} + I^{n,j}_{4} &\le \bigabs{\bigskp{ \bfS^n-  \widetilde \bfS}{ \bfD
				\bfu^{n,j} + \Rr(\bu^{n,j},\bpsi^{n,j}) }+\bigskp{
				\bfN^n- \widetilde \bfN}{\nabla \bfpsi^{n,j} } } 
		\\
		&\quad +\bigabs{\bigskp{ \bfS^n- \widetilde \bfS}{\!\nabla
				\bfu^{n,j}\chi_{\{\mathbf{u}^n\neq\mathbf{u}^{n,j}\}}}} \\
			&\quad +
		\bigabs{\bigskp{ \bfN^n- \widetilde \bfN}{\nabla
				\bfpsi^{n,j} \chi_{\{\bfpsi^n\neq \bfpsi^{n,j}\} }} } 
		\\
		& \quad + \bigabs{\bigskp{ \big(\bfS^n- \widetilde \bfS\big
				)^\anti }{\bfepsilon \cdot \bfpsi^{n,j}
				\chi_{\{\bfpsi^n\neq\bfpsi^{n,j}\} }}} \\
			&\quad + \bigskp{\bigabs{
				\big(\bfS^n- \widetilde \bfS\big )^\anti }}{\bigabs{
				\w^n-\w}\tau } 
		\\
		&\quad +\bigskp {\bigabs{ (\bfS^n- \widetilde \bfS)^\sym}}{\bigabs{(
				(\bfv^{n}-\bfv)\otimes \nabla \tau)^\sym }}\\
			&\quad  +\bigskp {\bigabs{
				\bfN^n-\widetilde \bfN}}{\bigabs{( \w^{n}-\w)\otimes
				\nabla \tau }} 
		\\
		&=: I_{5}^{n,j}+I_{6}^{n,j} +I_7^{n,j}+I_8^{n,j}+I_9^{n,j}+I_{10}^{n,j}+I_{11}^{n,j}\,.
\end{aligned}		\label{eq:i56}
	\end{align}\newpage
	\hspace*{-5mm}The term $I_5^{n,j}\!$, \hspace*{-0.1em}i.e.,
        \hspace*{-0.1em}the first line on the right-hand side in
        \eqref{eq:i56}, is handled~by \eqref{mon4}. For the other
        terms we~obtain,  using H\"older's inequality, \eqref{eq:mo-no}~and~\eqref{stab-est}, 
	\begin{align}
          I_6^{n,j}&\le 2\,\big (\Vert\bS^n\Vert_{p'(\cdot)} + \Vert\widetilde
                     \bS\Vert_{p'(\cdot)}
                     \big ) \Vert \nabla
                     \bfu^{n,j}\chi_{\{\mathbf{u}^n\neq\mathbf{u}^{n,j}\}} \Vert_{
                     L^{p(\cdot)}(B')} \notag
          \\[-1mm]
                   &\le 2\,K\Vert
                     \nabla \bfu^{n,j}\chi_{\{\mathbf{u}^n\neq\mathbf{u}^{n,j}\}}
                     \Vert_{ L^{p(\cdot)}(B')} \,,\label{eq:i8}
          \\[1mm]
          I_7^{n,j}&\le 2\,\big (\Vert\bN^n\Vert_{\smash{p'(\cdot),\vert \bfE\vert^{\smash{\frac{-2}{p(\cdot)-1}}}}} + \Vert\widetilde \bN\Vert_{\smash{p'(\cdot),\vert \bfE\vert^{\smash{\frac{-2}{p(\cdot)-1}}}}} \big ) \Vert \nabla
                     \bfpsi^{n,j}\chi_{\{\bfpsi^n\neq\bfpsi^{n,j}\}} \Vert_{
                     L^{p(\cdot)}(B';\abs{\bE}^2)} \notag
          \\[-1mm]
                   &\le  4\,K\,\|\bfE\|_{\infty}^{\frac{2}{p^-}}\Vert \nabla
                     \bfpsi^{n,j}\chi_{\{\bfpsi^n\neq\bfpsi^{n,j}\}} \Vert_{
                     L^{p(\cdot)}(B')} \,,\label{eq:i9}\\[1mm]
          I_{8}^{n,j}&\le 2\,\big (\Vert\hspace{-0.1em}(\bS^n)^\anti\Vert_{\smash{p'(\hspace{-0.1em}\cdot\hspace{-0.1em}),\vert \bfE\vert^{\smash{\frac{-2}{p(\hspace{-0.1em}\cdot\hspace{-0.1em})-1}}}}}\! + \!\Vert\widetilde \bS\,^\anti\Vert_{\smash{p'(\hspace{-0.1em}\cdot\hspace{-0.1em}),\vert \bfE\vert^{\smash{\frac{-2}{p(\hspace{-0.1em}\cdot\hspace{-0.1em})-1}}}}} \big ) \Vert 
                        \bfpsi^{n,j}\Vert_{L^\infty(B')}\abs{\Omega}^{\frac{1}{p^-}}\norm{\bE}_{\infty}^{\frac{2}{p^-}} \notag
          \\[-1mm]
                   &\le 2\, K \, \abs{\Omega}^{\frac{1}{p^-}}\norm{\bE}_{\infty}^{\frac{2}{p^-}}  \Vert \bfpsi^{n,j} \Vert_{ L^\infty(B')} \,,\label{eq:i10}\\[1mm]
          I_{9}^{n,j}&\le 2\,\big (\Vert(\bS^n)^\anti\Vert_{\smash{p'(\cdot),\vert \bfE\vert^{\smash{\frac{-2}{p(\cdot)-1}}}}} + \Vert\widetilde
                        \bS\,^\anti\Vert_{\smash{p'(\cdot),\vert \bfE\vert^{\smash{\frac{-2}{p(\cdot)-1}}}}} \big )
                        \Vert\w^n-\w \Vert_{L^{p(\cdot)}(B';\abs{\bE}^2)} \notag 
          \\[-1mm]
                   &\le  4\,K\,\|\bfE\|_{\infty}^{\frac{2}{p^-}}\Vert\w^n-\w \Vert_{
                     L^{p(\cdot)}(B')} \,,\label{eq:i11}\\[1mm]
          I_{10}^{n,j}&\le 2\, \big (\Vert\bS^n\Vert_{p'(\cdot)} + \Vert\widetilde
                        \bS \Vert_{p'(\cdot)}
                        \big )\norm{\nabla \tau}_\infty\,\Vert\bv^n- \bv
                        \Vert_{L^{p(\cdot)}(B')}\notag
          \\
                   &\le  2\,K\,\norm{\nabla \tau}_\infty\Vert\bv^n-\bv
                     \Vert_{L^{p(\cdot)}(B')} \,,\label{eq:i12}\\[1mm]
          I_{11}^{n,j}&\le 2\,\big (\Vert\bN^n\Vert_{\smash{p'(\cdot),\vert \bfE\vert^{\smash{\frac{-2}{p(\cdot)-1}}}}} + \Vert\widetilde \bN\Vert_{\smash{p'(\cdot),\vert \bfE\vert^{\smash{\frac{-2}{p(\cdot)-1}}}}} \big ) \norm{\nabla
                        \tau}_\infty\,\Vert\w^n-  \w \Vert_{
                        L^{p(\cdot)}(B';\abs{\bE}^2)}\notag 
          \\[-1mm]
                   &\le 4\, K\,\norm{\nabla
                     \tau}_\infty\,\|\bfE\|_{\infty}^{\frac{2}{p^-}}\Vert\w^n- \w
                     \Vert_{ L^{p(\cdot)}(B')} \,.\label{eq:i13}
	\end{align}
	Using \eqref{eq:C18px}, \eqref{stab-konv}--\eqref{konvergenz-stab}
	and $1\leq \lambda_{n,j}^p$, we deduce from \eqref{eq:I},
	\eqref{eq:est-ab}--\eqref{eq:i13} for any ${j\in \setN}$
	\begin{equation*}
                  \limsup_{n\to \infty} I^n\le \delta_j^\theta + c\,
                  K^{2\theta} 2^{-j(1-\theta)}+ c\,(1+
                  \norm{\bE}_{\infty}^{\frac{2}{p^-}}) ^{\theta} K^\theta
                  2^{\frac{-j\theta}{p^+}}\,.
	\end{equation*}
	Since $\lim_{j\to \infty} \delta_j  =0$, we observe that $I^n\!\to\! 0$ $(n\!\to\! \infty)$, which,~owing~to~${\theta \!\in\! (0,1)}$, (\hyperlink{(S.4)}{S.4}) and (\hyperlink{(N.4)}{N.4}),  implies for a suitable
	subsequence that
	\begin{align*}
		\big (\bS^n-\widetilde \bS\big ): \big
		(\D(\bv^n-\bv)+\Rr(\bv^n-\bv,\w^n-\w) \big) &\to 0
		\qquad \textrm{ a.e.~in }B&&\quad(n\to \infty)\,,
		\\
		\big(\bN^n-\widetilde \bN\big )\!:\!\big
		(\nabla\w^n-\nabla \w\big ) &\to 0 \qquad \textrm{ a.e.~in
		}B&&\quad(n\to \infty)\,.
	\end{align*}
	In view of \eqref{konvergenz-stab}, we also know that $\w ^n \to \w$
	a.e.~in~$B$~and,~hence,~we~can~conclude the assertion of
	Lemma \ref{DalMaso3p(x)} as in the proof of \cite[Lem.~6]{DalMaso2}.
\end{proof}

\begin{Kor}\label{cor:ae-conv}
	Let the assumptions of Lemma \ref{DalMaso3p(x)} be satisfied for all
	balls ${B\subset \subset \Omega_0}$ with
	${B':=2B \subset \subset \Omega_0}$.
	Then, we have for suitable subsequences that
	$\nabla \bv^n \to \nabla \bv$ a.e. in $\Omega$ $(n\to \infty)$,
	$ \hat\nabla \w^n \to \hat\nabla \w$ a.e.~in $\Omega$
	$(n\to \infty)$ and $\w^n \to \,\w $ a.e.~in $\Omega$
	$(n\to \infty)$.
\end{Kor}
\begin{proof}
	Using all rational tuples contained
	in~$\Omega_0$ as centers, we find a countable family $(B_k)_{k
		\in \setN}$ of balls covering $\Omega_0$ such that $B'_k:=2B_k
	\subset \subset \Omega_0$ for~every~${k \in \setN}$. Using the usual
	diagonalization procedure, we construct suitable subsequences
	such that $\w^n \to \w$ a.e.~in $\Omega_0$ $(n\to \infty)$, $\hat\nabla\w^n\to\hat\nabla\w$\footnote{Here, we used again that $(\hat\nabla\w)|_{\smash{B_k'}}=\nabla(\w|_{\smash{B_k'}})$ in $B_k'$ for all $k\in \mathbb{N}$ according to Remark~\ref{hatgrad}.} a.e.~in $\Omega_0$ $(n\to \infty)$~and
	$\nabla\bv^n\to\nabla\bv$ a.e.~in $\Omega_0$ $(n\to \infty)$. Since $\vert \Omega\setminus \Omega_0\vert=0$, we proved~the~assertion.
\end{proof}

\section{Main theorem}\label{sec:main}
Now we have everything at our disposal to prove our  main result,
namely the existence of \mbox{solutions} to the problem \eqref{NS}, \eqref{maxwell} for
$p^->\frac {2d}{d+2}$ even if the shear exponent $p$ is not globally
log--H\"older continuous. 
\begin{Sa}\label{thm:main4p(x)}
  Let $\Omega\subseteq\R^d$, $d\ge2$, be a bounded domain and let
  \mbox{Assumption~\ref{VssE}}, Assumption~\ref{VssSpx} and
  Assumption~\ref{VssNpx} be satisfied. If the shear exponent $p =\hat p \,\circ\,
  \abs{\bE}^2$ satisfies~${p^-\!>\!\frac{2d}{d+2}}$, then, for~any $\smash{\ff\hspace*{-0.05em}\in\hspace*{-0.05em} (\VDpo)^*}$ and
  $\smash{\bell\hspace*{-0.05em}\in \hspace*{-0.05em}(\XGpEo)^*}$, there exist \mbox{functions}
  $\smash{\mathbf{v}\in \VDpo}$ and $ \w\in \smash{\XGpEo }$ such that
  for every $\boldsymbol\varphi \in
  C^1_0(\Omega)$ with ${\divo \bphi = 0}$ and 
  $\boldsymbol\psi \in C_0^1(\Omega)$ with
  ${\nabla \boldsymbol\psi\in L^{\smash{\frac{q}{q-2}}}(\Omega;\vert
    \bfE\vert^{\smash{-\frac{\alpha q}{q-2}}})}$ for some
  ${q\in \left[1,(p^-)^*\right)}$, it holds
	\begin{align*}
			&\big\langle \Ss-\bv\otimes\bv,\bD\boldsymbol\varphi+\bR(\bphi,\boldsymbol\psi)\big\rangle
			\\
			& \;+\big\langle \bN(\hat \nabla \w ,\bE)-\w\otimes\bv,\nabla\boldsymbol\psi\big\rangle
			=\big\langle
			\ff,\boldsymbol\varphi\big\rangle+\big\langle\bell,
			\boldsymbol\psi\big\rangle \,.
	\end{align*}
	Moreover, there exists a constant $c>0$ such that 
	\begin{align*}
          &\Vert\bv\Vert_{{\VDpo}}+\Vert\w\Vert_{\XGpEo } 
          \\
          &\leq
		c \big (1+\norm{\bE}_2+\Vert
		\ff\Vert_{(\VDpo)^*}+\Vert\bell\Vert_{(\XGpEo)^*}\big
		)\,.
	\end{align*}
\end{Sa}

\begin{proof}
  \textbf{1. Non-degenerate approximation and a-priori estimates:}\\
  Resorting to standard pseudo-monotone operator theory, we deduce
  that for every $n\in \mathbb{N}$, there exist
  functions~$(\mathbf{v}^n,\boldsymbol\omega^n) \in
  (\smash{W^{1,p(\cdot)}_{0,\divo}(\Omega)}\cap L^r(\Omega)) \times
  (W^{1,p(\cdot)}(\Omega)\cap L^r(\Omega))$ satisfying for every
  ${\boldsymbol\varphi\in \smash{W^{1,p(\cdot)}_{0,\divo}(\Omega)}\cap
    L^r(\Omega)}$ and
  $\boldsymbol\psi\in W^{1,p(\cdot)}(\Omega)\cap L^r(\Omega)$
  \begin{align}
    \hspace*{-0.1cm}\begin{aligned}
      &\big\langle \Sn	-\bv^n\otimes\bv^n,\bD\boldsymbol\varphi+\bR(\bphi,\boldsymbol\psi)\big\rangle\\&\quad+\tfrac 1n
      \big\langle\vert\nabla\mathbf{v}^n\vert^{p(\cdot)-2}\nabla\mathbf{v}^n,\nabla\bphi\big\rangle +\tfrac 1n
      \big\langle\big(\vert\mathbf{v}^n\vert^{p(\cdot)-2}+\vert\mathbf{v}^n\vert^{r-2}\big)\mathbf{v}^n,\bphi\big\rangle 
      \\
      & \quad+\big\langle\bN(\nabla \bomega^n,\bE)-\w^n\otimes\bv^n,\nabla\boldsymbol\psi\big\rangle \\&\quad +\tfrac{1}{n}\big\langle\vert\nabla
      \boldsymbol{\omega}^n\vert^{{p(\cdot)-2}{}}\nabla\boldsymbol{\omega}^n,\nabla\boldsymbol{\psi}\big\rangle+\tfrac 1n
      \big\langle\big(\vert{\w}^n\vert^{p(\cdot)-2}+
      \vert{\w}^n\vert^{r-2}\big){\w}^n,\bpsi\big\rangle\\&
      =\big\langle
      \ff,\boldsymbol\varphi\big\rangle+\big\langle\bell,\boldsymbol\psi\big\rangle\,,
    \end{aligned} \label{NS2}
  \end{align}
  where\footnote{We have chosen $r$ such that the convective terms
    $\skp{\bv\otimes \bv}{\nabla \bphi}$ and
    $\skp{\w\otimes \bv}{\nabla \bphi}$ define compact operators from
    $L^r(\Omega)\times L^r(\Omega)$ to $\smash{(W^{1,p^-}_0(\Omega))^*}$.}
  $r>2(p^-)'$. Moreover, there exists a constant $c>0$ (independent~of~${n\in \mathbb{N}}$) such that
  \begin{align}
    &\tfrac{1}{n} \big (\, \rho_{p(\cdot)}(\nabla \bfv^n) +\rho_{p(\cdot)}(\bfv^n) +
      \Vert \bfv^n\Vert_r^{r}+\rho_{p(\cdot)}(\nabla \w^n) +\rho_{p(\cdot)}(\w^n) +
      \Vert \w^n\Vert_r^{r}\,\big )\label{apri0}
    \\& 
    \quad+\rho_{p(\cdot)}(\bD \bfv^n) +\rho_{p(\cdot), \abs{\bE}^2}(\bfR(\bfv^n,\w^n)) +
 \rho_{p(\cdot), \abs{\bE}^2}(\nabla \w^n)\leq K(\bE,\ff,\bell)=:K\,. \notag
  \end{align}
  Under the usage of Lemma \ref{lem:unit_ball_px}, also using Korn's
  and Poincar\'e's inequality for the constant exponent
  $p^->\frac{2d}{d+2}$, we deduce from \eqref{apri0} that
	\begin{align}
		\begin{aligned}
		&
		\tfrac{1}{n}\big(\Vert \mathbf{v}^n\Vert_r+\Vert \mathbf{v}^n\Vert_{1,p(\cdot)}+\Vert\boldsymbol\omega^n\Vert_r+ \Vert\boldsymbol\omega^n\Vert_{1,p}\big)\\&\quad+\Vert\bv^n\Vert_{{\VDpo}}+\|\bfR(\bfv^n,\w^n)\|_{p(\cdot),\vert\bfE\vert^2}+\Vert\w^n\Vert_{\XGpEo }
		 \leq K\,.\end{aligned}\label{apri}
	\end{align}
       Using (\hyperlink{(S.2)}{S.2}),
        (\hyperlink{(N.2)}{N.2}),~\eqref{apri}, Assumption \ref{VssE}
        and the notation introduced~in~\eqref{eq:not}, we obtain
	\begin{equation}
		\label{eq:est-SNn}
		\begin{aligned}\Vert\mathbf{S}^n\Vert_{p'(\cdot)}+
			\Vert (\mathbf{S}^n)^\anti\Vert_{\smash{p'(\cdot),\vert \bfE\vert^{\smash{\frac{-2}{p(\cdot)-1}}}}}+
			\Vert\mathbf{N}^n\Vert_{\smash{p'(\cdot),\vert \bfE\vert^{\smash{\frac{-2}{p(\cdot)-1}}}}}\le K\,.
		\end{aligned}
	\end{equation}\newpage
	
	\textbf{2. Extraction of (weakly) convergent subsequences:}\\
	The estimates \eqref{apri}, \eqref{eq:est-SNn} and Remark \ref{hatgrad} (ii) yield~not~relabeled~subsequences  as well as  functions 
	$\smash{\bv\in \VDpo}$, $\smash{\w\in \XGpEo
	}$, $\smash{\widehat \bS \in L^{p'(\cdot)}(\Omega)}$ and
	$\smash{\widehat \bN} \in L^{p'(\cdot)}(\Omega;\vert \bfE\vert^{\smash{-\frac{2}{p(\cdot)-1}}})$ such that
	\begin{align}
		\bv^n&\rightharpoonup\bv\quad& &\text{in}\
		\VDpo&&\quad(n\to \infty)\,,\hphantom{\;}\notag
		\\[-1mm]
		\w^n&\rightharpoonup\w\quad& &\text{in}\ \XGpEo &&\quad(n\to \infty)\,,\label{konvergenz}\\[-0.5mm]
		\bfR(\bfv^n,\w^n)&\rightharpoonup\bfR(\bfv,\w)\quad& &\text{in}\ L^{p(\cdot)}(\Omega;\vert \bfE\vert^2) &&\quad(n\to \infty)\,,\notag
		\\[1mm]
		\mathbf{S}^n&\rightharpoonup \widehat \bS \;& &\text{in}\
		L^{p'(\cdot)}(\Omega)&&\quad(n\to \infty)\,,\notag
		\\[-1mm]
		\hspace*{-1.75em}(\mathbf{S}^n)^{\anti}&\rightharpoonup \widehat \bS^\anti & &\text{in}\ L^{p'(\cdot)}(\Omega;\vert\bfE\vert^{\frac{-2}{p(\cdot)-1}})&&\quad(n\to \infty)\,,\label{konvergenz18}
		\\[-1mm]
		\mathbf{N}^n&\rightharpoonup \widehat\bN \quad& &\text{in}\
		L^{p'(\cdot)}(\Omega;\vert\bfE\vert^{\frac{-2}{p(\cdot)-1}})&&\quad(n\to \infty)\,.\notag
	\end{align}

	\textbf{3. Identification of $\widehat\bS$ with $\Ss$~and~$\widehat\bN$~with~${\bfN(\hat\nabla\w,\bE)}$:}\\
	Recall that $\Omega_0\!=\! \{ x\!\in \!\Omega \hspace*{-0.1em}\mid\hspace*{-0.1em} \vert
	\mathbf{E}(x)\vert\!>\!0\}$ (cf. Assumption \ref{VssE}) and let
        ${B\!\subset\!\subset\! \Omega_0}$~be~a~ball such that
        $B'\!:=\!2B \!\subset
        \subset\!\Omega_0$. Then,
        due~to~Remark \ref{locallog}, Lemma~\ref{symgradp(x)}~and~Lemma~\ref{hatgrad}, we have that $\smash{\VDpo,\XGpEo} \hookrightarrow \! W^{1,p(\cdot
		)}(B')$. \!Therefore,~from~$\eqref{konvergenz}_{1,2}$,
              Rellich's compactness theorem and Theorem \ref{rellichp(x)},~we~deduce~that
	\begin{align}
		\begin{aligned}\label{konvergenz-w}
			\bfv^n&\rightharpoonup\bfv\quad& &\text{in}\  W^{1,p(\cdot)}(B')&&\quad(n\to \infty)\,,
			\\
			\bfv^n&\to\bfv\quad& &\text{in}\
                        L^{p(\cdot)}(B') \cap L^{q}(\Omega) \text{ and a.e. in }
			B'&&\quad(n\to \infty)\,,\\
			\w^n&\rightharpoonup\w\quad& &\text{in}\  W^{1,p(\cdot)}(B')&&\quad(n\to \infty)\,,
			\\
			\w^n&\to\w\quad& &\text{in}\  L^{p(\cdot)}(B') \cap L^{q}(\Omega) \text{ and a.e. in }
			B'&&\quad(n\to \infty)\,,
		\end{aligned}
	\end{align}
where $q\in \left[1,(p^-)^*\right)$. 	Next, let $\tau\in C_0^\infty(B')$ be such that $\chi_B \leq \tau\leq\chi_{B'}$.
	Due to
(\ref{konvergenz-w})$_{1,3}$,~it~follows~that
	\begin{align}
		\begin{aligned}\label{konvergenz-w1}
			\mathbf{u}^n&:=(\mathbf{v}^n-\mathbf{v})\tau\rightharpoonup\mathbf{0}
			& &\text{in}\ W^{1,p(\cdot)}_0(B')&&\quad(n\to \infty)\,,
			\\
			\bfpsi^n&:=(\boldsymbol\omega^n-\boldsymbol\omega)\tau
			\rightharpoonup\mathbf{0} & &\text{in}\
			W^{1,p(\cdot)}_0(B')&&\quad(n\to \infty)\,.
		\end{aligned}
	\end{align}
	Denote for $n\in \mathbb{N}$, the Lipschitz truncations  of
        ${\bu^n,\bfpsi^n\in  W^{1,p(\cdot)}_0(B')}$
	according to Theorem \ref{thm:Ltp(x)} with respect to $\hspace*{-0.1em}B'\hspace*{-0.1em}$ by $(\mathbf{u}^{n,j})_{j\in \mathbb{N}}, \hspace*{-0.1em}{(\bfpsi^{n,j})_{j\in \mathbb{N}}\!\subseteq\! W^{1,\infty}_0(B')}$.~\hspace*{-0.1em}In~\hspace*{-0.1em}\mbox{particular}, based on  \eqref{konvergenz-w1}, Theorem \ref{thm:Ltp(x)} implies that these Lipschitz truncations satisfy for
	every $j \in \setN$~and~${s \in [1,\infty)}$ 
	\begin{align}
		\begin{aligned}\label{Testfkt_2}
			\mathbf{u}^{n,j}&\rightharpoonup \mathbf{0}& &\text{in}\
			W^{1,s}_0(B')&&\quad(n\to \infty)\,,
			\\
			\mathbf{u}^{n,j}&\to \mathbf{0}& &\text{in}\
			L^{s}(B')&&\quad(n\to \infty)\,,
			\\
			\bfpsi^{n,j}&\rightharpoonup \mathbf{0}& &\text{in}\ W^{1,s}_0(B')&&\quad(n\to \infty)\,,
			\\
			\bfpsi^{n,j}&\to \mathbf{0}& &\text{in}\ 
			L^{s}(B')&&\quad(n\to \infty)\,.
		\end{aligned}
	\end{align}
	Note that $\bfpsi^{n,j}\in
	W^{1,\infty}_0(B')$, $n,j\hspace*{-0.12em}\in\hspace*{-0.12em} \mathbb{N}$, are
	suitable test-functions~in~\eqref{NS2}.~\mbox{However}, 
	$\mathbf{u}^{n,j}\hspace*{-0.1em}\in\hspace*{-0.1em}
	W^{1,\infty}_0(B')$,
	$n,j\hspace*{-0.1em}\in\hspace*{-0.1em} \mathbb{N}$, are not
	admissible in \eqref{NS2} as they are not
	di\-vergence-free.  To correct this,  we
	define 
	${ \bw^{n,j}:=\mathcal{B}_{B'}(\divo
		\mathbf{u}^{n,j})\in
		W^{1,s}_{0,\divo}(B')}$,~for~${n,j\in
		\mathbb{N}}$, where~$\smash{\mathcal{B}_{B'}:L^s_0(B')\to W^{1,s}_0(B')}$
	is the Bogovskii operator with respect~to~$B'$.  Since
	$\mathcal{B}_{B'}$ is weakly continuous, \eqref{Testfkt_2}$_{1,2}$ and
	Rellich's compactness~theorem~\mbox{imply} for
	every~${j\in \setN}$~and~${s \in (1, \infty)}$ that
	\begin{align}
		\begin{aligned}\label{Testfkt_3}
			\bw^{n,j}&\rightharpoonup \mathbf{0} & &\text{in}\ W^{1,s}_0(B')&&\quad(n\to \infty)\,,
			\\
			\bw^{n,j}&\to \mathbf{0} & &\text{in}\
			L^s(B')&&\quad(n\to \infty)\,.
		\end{aligned}
	\end{align}
	Apart from that, owing to the boundedness of $\mathcal{B}_{B'}$, since $p|_{B'}\in \mathcal{P}^{\log}(B')$ holds, (cf. Proposition \ref{bogp(x)}), one has for any~${n,j\in \setN}$ that
	\begin{equation}
		\label{eq:divo}
		\begin{aligned}
			\Vert\bw^{n,j}\Vert_{W^{1,p(\cdot)}_0(B')}&\leq
			c\,\Vert\Div\mathbf{u}^{n,j}\Vert_{L^{p(\cdot)}(B')}\,.
		\end{aligned}
	\end{equation}
	On the  basis of $\nabla \bfu^n = \nabla \bfu^{n,j}$ on the set $\set{\bfu^n =
		\bfu^{n,j} }$ (cf.~\cite[Cor.~1.43]{maly-ziemer})~and $\divo
	\bfu^n =\nabla \tau \cdot (\bv^n -\bv)$ for every~${n,j\in \setN}$, we further get for every~${n,j\in \setN}$ that
	\begin{equation}	\label{eq:divo2}
		\divergence \bfu^{n,j} = \chi_ { \set{\bfu^n \not= \bfu^{n,j} }}
		\divergence \bfu^{n,j} + \chi_ { \set{\bfu^n = \bfu^{n,j} }}\nabla
		\tau \cdot (\bv^n -\bv)\quad\text{ a.e. in }B'\,.
	\end{equation}
	Then, \eqref{eq:divo} and 	\eqref{eq:divo2} together imply for every~${n,j\in \setN}$
	\begin{align*}
		\norm{\bw^{n,j}}_{W^{1,p(\cdot)}_0(B')}\leq c\,
		\norm{\nabla \bfu^{n,j}\, \chi_ { \set{\bfu^n \not= \bfu^{n,j} } }
		}_{L^{p(\cdot)}(B')}  + c\,(\norm{\nabla \tau}_\infty ) \norm{\bv^n -\bv}_{L^{p(\cdot)}(B')}
		\,,
	\end{align*}
	which in conjunction with \eqref{eq:C18px} and \eqref{konvergenz-w}$_2$  yields for every $j \in \setN$ that
	\begin{align}
		\begin{aligned}
			\smash{\limsup_{n\to \infty} \|\bw^{n,j}\|_{W^{1,p(\cdot)}_0(B')} \leq c\,2^{\frac {-j}{p^+}}}\,.
		\end{aligned}
		\label{eq:n3.29*a}
	\end{align}
	Setting $\boldsymbol\varphi^{n,j}:=\mathbf{u}^{n,j}-\bw^{n,j}$ for all~${n,j\in \setN}$,
	we observe that $(\boldsymbol\varphi^{n,j})_{n,j\in \setN}$
          belong to $V_s(B') \cap V_{p(\cdot)}(B')$, $s \in
          (1,\infty)$, i.e., they are suitable test-functions in
	\eqref{NS2}. To use Corollary \ref{cor:ae-conv}, we  have to verify that  condition
	\eqref{mon4} is satisfied.~To~this~end, we test equation (\ref{NS2})
	with the admissible test-functions $\bphi=\boldsymbol\varphi^{n,j}$ and 
	$\bpsi=\bfpsi^{n,j}$ for every ${n,j\in \setN}$ and subtract on both sides 
	\begin{align*} 
		\big\langle\Ss, \D\bu^{n,j}+\Rr(\bu^{n,j},\bfpsi^{n,j})\big\rangle+
		\big\langle\bN(\hat\nabla\w,\bE),\nabla\bfpsi^{n,j}\big\rangle\,, \quad n,j\in \setN\,. 
	\end{align*} 
	Owing to $\bphi^{n,j}=\bu^{n,j}-\bw^{n,j}$ for every ${n,j\in \setN}$, this yields~for~every~${n,j\in \setN}$~that 
	\begin{align}
		&\big\langle\bfS^n-\Ss,\,
		\D\bu^{n,j}+\Rr(\bu^{n,j},\bfpsi^{n,j})\big\rangle \notag
		+ \big\langle\bfN^n-\bN(\hat\nabla\w,\bE),\nabla\bfpsi^{n,j}\big\rangle \notag 
		\\
		&= \big\langle \ff,\boldsymbol\varphi^{n,j}\big\rangle +
		\big\langle\bell,\bfpsi^{n,j}\big\rangle\notag 
		\\&\quad-\tfrac 1n
		\big\langle\vert\nabla\mathbf{v}^n\vert^{p(\cdot)-2}\nabla\mathbf{v}^n,\nabla \bfphi^{n,j}\big\rangle-\tfrac 1n
		\big\langle\big(\vert{\bfv}^n\vert^{p(\cdot)-2}+
		\vert{\bfv}^n\vert^{r-2}\big){\bfv}^n,\bfphi^{n,j}\big\rangle\notag\\&\quad -\tfrac{1}{n}\big\langle\vert\nabla
		\boldsymbol{\omega}^n\vert^{{p(\cdot)-2}{}}\nabla\boldsymbol{\omega}^n,\nabla\bfpsi^{n,j}\big\rangle-\tfrac 1n
		\big\langle\big(\vert{\w}^n\vert^{p(\cdot)-2}+
		\vert{\w}^n\vert^{r-2}\big){\w}^n,\bfpsi^{n,j}\big\rangle \notag
		\\
		&\quad
		+
		\big\langle\bv^n\otimes\bv^n,\nabla\bfphi^{n,j}\big\rangle
		+\big\langle\w^n\otimes\bv^n,\nabla\bfpsi^{n,j}\big\rangle\notag\\&\quad
		+\big\langle \Sn,\nabla\bw^{n,j}\big\rangle \notag
		\\
		&\quad -\big\langle\Ss, \D\bu^{n,j}+\Rr(\bu^{n,j},\bfpsi^{n,j})\big\rangle
		-\big\langle\bN(\hat\nabla\w,\bE),\nabla\bfpsi^{n,j}\big\rangle \notag
		\\
		&=:
		{\sum}_{k=1}^{11} J_k^{n,j}\,.\label{eq:diff1}
	\end{align}

	\noindent Because of $\bv\in \smash{\VDpo}$,
	$\w\in \smash{\XGpEo} $ and $\bR(\bv , \w) \in L^{p(\cdot)}(\Omega;\abs{\bE}^2)$, we
	obtain
        using~(\hyperlink{(S.2)}{S.2})~and~(\hyperlink{(N.2)}{N.2})~that
        $\bN(\hat\nabla\w,\bE) \in \smash{L^{p'(\cdot)}(\Omega;\vert
          \bfE\vert^{\smash{\frac{-2}{p(\cdot)-1}}})}$ and $\Ss \in  L^{p'(\cdot)}(\Omega)$~(cf.~\eqref{eq:est-SNn}). Using this, (\ref{Testfkt_2}) and
	(\ref{Testfkt_3}), we conclude for every $j \in \setN$ that
	\begin{align}\label{eq:j1} 
		\lim_{n\to\infty} J_1^{n,j} +J_2^{n,j} +J_{10}^{n,j} +J_{11}^{n,j} =0\,.
	\end{align}
	From  (\ref{apri}),
	(\ref{Testfkt_2}) and (\ref{Testfkt_3}), we obtain for every $j \in
	\setN$ that
	\begin{align}\label{eq:j3} 
		\lim_{n\to\infty} J_3^{n,j}+J_4^{n,j} +J_5^{n,j} +J_6^{n,j} =0\,.
	\end{align}
	Recalling \eqref{eq:not} we get, in view of \eqref{eq:est-SNn}
	and
	\eqref{eq:n3.29*a}, that for every $j\in \setN$,~it~holds
	\begin{align}
		\limsup_{n\to\infty} J_9^{n,j} \leq \limsup_{n\to\infty}\Vert \bfS^n\Vert_{p'(\cdot)} \Vert
		\nabla\bw^{n,j}\Vert_{L^{p(\cdot)}(B')}
		\leq
		K\, 2^{\frac
			{-j}{p^+}}=:\delta_j\,.\label{eq:j8}
	\end{align} 
	From 
	\eqref{konvergenz-w}$_{2,4}$, it further  follows that %
	\begin{equation}
		\begin{split}\label{conv-convec} 
			\begin{aligned}
				\mathbf{v}^n\otimes \mathbf{v}^n&\to \mathbf{v}\otimes
				\mathbf{v}&& \text{in}\ L^{s'}(B')&&\quad(n\to \infty)\,,
				\\
				\w^n\otimes \mathbf{v}^n&\to \w\otimes \mathbf{v}&&\text{in}\ L^{s'}(B')&&\quad(n\to \infty)\,,
			\end{aligned}
			\hspace*{5mm}
			\begin{aligned}
				s'\in \Big[1, \frac{(p^-)^*}{2}\Big)\,.
			\end{aligned}
		\end{split}
	\end{equation}
	Thus, combining (\ref{Testfkt_2}), (\ref{Testfkt_3}) and \eqref{conv-convec}, we find that for every $j \in \setN$
	\begin{align}\label{eq:j4} 
		\lim_{n\to\infty} J_7^{n,j} +J_8^{n,j}  =0\,.
	\end{align}
	From \eqref{eq:diff1}--\eqref{eq:j4} follows \eqref{mon4}. Thus, Corollary \ref{cor:ae-conv} yields subsequences~with
	\begin{align}
		\begin{aligned}
			\nabla\bv^n&\to\nabla\bv&&\quad\text{ a.e. in }\Omega\,,\\[-0.5mm] 
			\hat\nabla\w^n&\to\hat\nabla\w&&\quad\text{ a.e. in }\Omega\,,\\
			\w^n&\to\,\w&&\quad\text{ a.e. in }\Omega\,.\end{aligned}\label{pw}
	\end{align}
	Since $\mathbf{S}\in C^0( \mathbb{R}^{d\times d}_{\textup{sym}}\times \mathbb{R}^{d\times d}_{\textup{skew}}\times \mathbb{R}^d;\mathbb{R}^{d\times d})$ (cf.~(\hyperlink{(S.1)}{S.1})) and $\mathbf{N}\in C^0(\mathbb{R}^{d\times d}\times \mathbb{R}^d;\mathbb{R}^{d\times d})$ (cf.~(\hyperlink{(N.1)}{N.1})), we
	deduce from \eqref{pw} that
	\begin{align}\label{eq:aen}
		\begin{aligned}
			\bS^n&\to \Ss& &\text{a.e. in}\ \Omega&&\quad(n\to \infty)\,,
			\\
			\bN^n&\to  \bN(\hat\nabla\w,\bE)& &\text{a.e. in}\ \Omega&&\quad(n\to \infty)\,.
		\end{aligned}
	\end{align}
	To identify $\widehat \bS$, we now argue as in the proof of
	\mbox{\cite[Thm.~4.6  \!(cf.  \!(4.21)$_1$--(4.23)$_1$)]{erw}}, while Theorem
	\ref{pfastue} (with $G=\Omega$ and $\sigma=\abs{\bE}^2$),
	\eqref{konvergenz18}, \eqref{eq:aen} and the absolute continuity of
	Lebesgue measure with respect to the measure $\nu_{\abs{\bE}^2}$ is~used~to~identify~$\widehat \bN$. Thus, we just proved
	\begin{align}
		\begin{aligned}\label{eq:SN}
			\widehat \bS=\Ss \quad \text {and} \quad 
			\widehat \bN=\bN(\hat\nabla\w,\bE)\,.
		\end{aligned}
	\end{align}
	Now we have at our disposal everything to identify the limits of all
	but~one~term in \eqref{NS2}. Using \eqref{apri},  \eqref{konvergenz}, \eqref{konvergenz18},
	\eqref{conv-convec}$_1$,  \eqref{eq:SN} as well as  $p^-> \frac {2d}{d+2}$, we~obtain from \eqref{NS2}
	that  for every
	$\boldsymbol\varphi \in C^1_0(\Omega)$ with $\divo
	\bphi =0$ and for every $\boldsymbol\psi \in C_0^1(\Omega)$,~it~holds
	\begin{align}\begin{aligned}
			&\big\langle \Ss-\bv\otimes\bv,\bD\bphi+\bR(\bphi,\boldsymbol\psi)\big\rangle 
			\\
			& \;+\big\langle\bN(\hat\nabla\w,\bE),\nabla\boldsymbol\psi\big\rangle -\lim _{n\to
				\infty
			}\big\langle\w^n\otimes\bv^n,\nabla\boldsymbol\psi\big\rangle
			=\big\langle
			\ff,\boldsymbol\varphi\big\rangle+\big\langle\bell,\boldsymbol\psi\big\rangle
			\,.
		\end{aligned}\label{NS2aa}
	\end{align}
	Finally, we have to check whether the remaining
        limit~in~\eqref{NS2aa}~exists~and  identify it. To this end, we
        fix an arbitrary $\boldsymbol\psi \in C_0^1(\Omega)$ with
        $\nabla \boldsymbol\psi\in
        L^{\smash{\frac{q}{q-2}}}(\Omega;\vert
        \bfE\vert^{\smash{-\frac{\alpha q}{q-2}}})$ for some $q \in
        \left[1,(p^-)^*\right)$ and choose  $\Omega'$ such that $\textup{int}(\textup{supp}(\boldsymbol\psi)) \subset\subset\Omega' \subset\subset\Omega$. Due to Theorem~\ref{compactnew} and $\eqref{konvergenz}_3$, it holds
	\begin{align*}
		\w^n\rightharpoonup \w\quad\text{ in }L^q(\Omega';\vert\bE\vert^{\alpha q})\quad(n\to \infty)
	\end{align*}
	for every $\alpha\hspace*{-0.1em}\ge \hspace*{-0.1em}1+\frac{2}{p^-}$. On the other hand, due to $\nabla \boldsymbol\psi\hspace*{-0.1em}\in\hspace*{-0.1em} L^{\smash{\frac{q}{q-2}}}(\Omega;\vert \bfE\vert^{\smash{-\frac{\alpha q}{q-2}}})$~and~$\eqref{konvergenz}_2$, using H\"older's inequality, 
	we also see that
	\begin{align*}
		\nabla\boldsymbol\psi\bv^n\rightarrow\nabla\boldsymbol\psi\bv\quad\text{ in }L^{q'}(\Omega';\vert\bE\vert^{\frac{-\alpha q}{q-1}})\quad(n\to \infty)\,.
	\end{align*}
	Since $(L^q(\Omega',\vert\bE\vert^{\alpha
          q}))^*\hspace*{-0.1em}\simeq \hspace*{-0.1em}L^{q'}(\Omega',\vert\bE\vert^{\frac{-\alpha
            q}{q-1}})$, we infer that
        $$
        \lim _{n\to \infty
          }\big\langle\w^n\otimes\bv^n,\nabla\boldsymbol\psi\big\rangle =
          \big\langle\w\otimes\bv,\nabla\boldsymbol\psi\big\rangle\,,
          $$
        which looking back to \eqref{NS2aa} concludes the proof of
        Theorem~\ref{thm:main4p(x)}.
\end{proof}

\appendix

\section{Proof of Theorem \ref{poincarep(x)}}\label{sec:app}

An essential ingredient in the proof of Theorem \ref{poincarep(x)}   is the following  Gagliardo--Nirenberg interpolation inequality.

\begin{Lem}\label{gagliardo-nirenberg}
	Let $\Omega\subseteq \mathbb{R}^d$, $d\ge 2$, be a bounded $C^1$--domain and $s,r \in\left[1,\infty\right)$. Then, for every ${\bfu\in  W^{1,r}(\Omega)\cap L^s(\Omega)}$,  it holds ${\bfu\in L^q(\Omega)}$~with
	\begin{align}
		\smash{\| \bfu\|_q\leq c\,\| \bfu\|_{1,r}^\theta\|\bfu\|_s^{1-\theta}\,,}\label{GN}
	\end{align}
	where $c=c(q,r,s,\Omega)>0$, provided that $\frac{1}{q}=\theta(\frac{1}{r}-\frac{1}{d})+(1-\theta)\frac{1}{s}$~for~some~${\theta\in \left[0,1\right]}$, unless $r=d$, in which \eqref{GN} only holds for $\theta\in \left[0,1\right)$.
\end{Lem}

\begin{proof}
	For the case $r\neq d$, we can refer to
        \cite[Thm.~10.1]{Fri08}.~For~the~case~${r=d}$,~we can refer to
        \cite{Gag59} or \cite{Nir59}, where \eqref{GN} is proved with
        an additional $\|\bfu\|_s$ on the right-hand side.  To get rid
        of this additional term 
        we~can~use that $\|\bfu\|_s=
          \|\bfu\|_s^\theta\|\bfu\|_s^{1-\theta}$ $\leq
          c\,\|\bfu\|_{1,r}^\theta\|\bfu\|_s^{1-\theta}$.
\end{proof}

Since the Gagliardo--Nirenberg interpolation inequality only takes into account the full gradient, but we only have control over the symmetric part of the gradient, we need to switch locally from the full gradient to the symmetric gradient via Korn's second inequality.

\begin{Lem}\label{korn2}
	Let $\Omega\subseteq\mathbb{R}^d$, $d\ge 2$, be a bounded Lipschitz domain and $p\in \left(1,\infty\right)$. Then, there exists a constant $c=c(p,\Omega)>0$ such that for every $\bfu\in X^{p,p}_{\bD}(\Omega)$,  it holds ${\bfu\in W^{1,p}(\Omega)}$ with
	\begin{align*}
		\|\bfu\|_{1,p}\leq c\,\| \bfu\|_{X^{p,p}_{\bD}(\Omega)}\,.
	\end{align*}
\end{Lem}

\begin{proof}
	See \cite[Thm.~5.1.10, (5.1.17)]{mnrr}.
\end{proof}

Aided  by Lemma~\ref{gagliardo-nirenberg} and Lemma~\ref{korn2}, we can next prove 
Theorem \ref{poincarep(x)}.

\begin{proof}[Proof of Theorem \ref{poincarep(x)}]
	The proof is relies on techniques from \cite[Lem.~3.20]{alex-diss}. We split the proof into two steps:
	
	\textbf{Step 1:} First, let $\bfu\in \mathcal{V}^{p^-,p(\cdot)}_{\textbf{D}}$ be arbitrary.
	Since $p\in C^0(\Omega)$ and $\Omega'\subset\subset \Omega$, there exists a finite 
	 covering of $ \overline{\Omega'}$ by open balls ${(B_i)_{i=1,\dots,m}}$, ${m\hspace*{-0.1em}\in\hspace*{-0.1em}\mathbb{N}}$,~with~${B_i\hspace*{-0.1em}\subset\subset \hspace*{-0.1em}\Omega}$, $i\hspace*{-0.15em}=\hspace*{-0.15em}1,\dots,m$, such that the local exponents 
	$p^+_i\hspace*{-0.2em}=\hspace*{-0.15em}\sup_{x\in B_i}{p(x)}$ and ${p^-_i\hspace*{-0.2em}=\hspace*{-0.15em}\inf_{x\in B_i}{p(x)}}$ satisfy for every ${i=1,\dots,m}$
	\begin{align}
		p^+_i<p^-_i\Big(1+\frac{2}{d}\Big)\,.\label{eq:2.30.1}
	\end{align}
	In addition, there exists some $\Omega''\subset\subset \Omega$
        with $\partial \Omega''\in C^1$ and $\bigcup
        _{i=1}^mB_i\subset \subset \Omega''$. 
	Let us fix an arbitrary ball $B_i$ for some
        $i=1,\dots,m$. There are two possibilities. First, we consider
        the case ${p_i^+\leq 2}$. Using $a^{p(x)}\leq (1+a)^2\leq
        2+2a^2$, valid for all $a\ge 0$ and $x\in B_i$, and ${\|\bfu\|_{L^2(B_i)}\leq\|\bfu\|_{L^2(\Omega'')}}$, where we also used that $\bfu \in L^2(\Omega'')$ since $\bfu\in C^\infty(\overline{\Omega''})$, we observe that
	\begin{align}
		\rho_{p(\cdot)}(\bfu\chi_{B_i})\leq 2\,\vert B_i\vert +2\|\bfu\|_{L^2(B_i)}^2\leq 2\,\vert \Omega\vert +2\,\|\bfu\|_{L^2(\Omega'')}^2\,.\label{eq:2.30.2}
	\end{align}
	Next, assume that $\smash{p_i^+> 2}$. Then, exploiting that $\smash{p_i^->
		p_i^+\frac{d}{d+2}>
		\frac{2d}{d+2}}$~(cf.~\eqref{eq:2.30.1}), i.e.,~$\frac{d-\smash{p_i^-}}{dp_i^-}<\frac{1}{2}$, we deduce that
	\begin{align}
		0<\theta_i&:=\frac{\frac{1}{2}-\frac{1}{p_i^+}}{\frac{1}{2}-\frac{d-p_i^-}{dp_i^-}}=\frac{p_i^-}{p_i^+}\frac{d(p_i^+-2)}{d(p_i^--2)+2p_i^-}\hspace*{-0.1cm}\overset{\eqref{eq:2.30.1}}{<}\hspace*{-0.075cm}\frac{p_i^-}{p_i^+}\frac{d(p_i^-+\frac{2p_i^-}{d}-2)}{d(p_i^--2)+2p_i^-}
		=\frac{p_i^-}{p_i^+}
		\leq 1\,.\label{eq:2.30.3}
	\end{align}
	Owing to $\smash{\partial B_i\in C^\infty}$, Korn's second inequality (cf.~Lemma~\ref{korn2}) yields a constant $c_i=c_i(p^-_i,B_i)>0$ such that
	\begin{align}
		{\|\bfu\|_{\smash{L^{p^-_i}(B_i)}}+\|\nabla\bfu\|_{\smash{L^{p^-_i}(B_i)}}\leq c_i\big(\|\bfu\|_{\smash{L^{p^-_i}(B_i)}}+\|\bfD\bfu\|_{\smash{L^{p^-_i}(B_i)}}\big)\,.}\label{eq:2.30.5}
	\end{align}
	Thanks to the Gagliardo--Nirenberg interpolation inequality (cf.~Lemma~\ref{gagliardo-nirenberg}), since ${\theta_i\in  \left(0,1\right)}$  satisfies $\frac{1}{p_i^{\smash{+}}}=\theta_i(\frac{1}{p_i^{\smash{-}}}-\frac{1}{d})+(1-\theta_i)\frac{1}{2}$,
	there exists a  constant
        $\smash{c_i=c_i(p^-_i,p_i^+,\Omega')}$
        such that 
	\begin{align}
		\smash{\rho_{p_i^+}(\bfu\chi_{B_i})\leq c_i\big(\|\bfu\|_{\smash{L^{p^-_i}(B_i)}}+\|\nabla\bfu\|_{\smash{L^{p^-_i}(B_i)}}\big)^{p_i^+\theta_i}\|\bfu\|^{p_i^+(1-\theta_i)}_{L^2(B_i)}\,.}\label{eq:2.30.4}
	\end{align}
	By inserting \eqref{eq:2.30.5} in \eqref{eq:2.30.4},~we~get~that
	\begin{align}
		\begin{split}
			\rho_{p_i^+}(\bfu\chi_{B_i})\leq c_i\big(\|\bfu\|_{\smash{L^{p_i^-}(B_i)}}+\|\bfD\bfu\|_{\smash{L^{p_i^-}(B_i)}}\big)^{p_i^+\theta_i}\|\bfu\|_{L^2(B_i)}^{p_i^+(1-\theta_i)}\,.
		\end{split}\label{eq:2.30.6}
	\end{align}
	Since $p_i^+\theta_i\hspace*{-0.14em}<\hspace*{-0.14em}p_i^-$ (cf.~\eqref{eq:2.30.3}), we can apply the $\varepsilon$--Young~inequality with~respect~to  $\rho_i\hspace*{-0.1em}:=\hspace*{-0.1em}p_i^-(p_i^+\theta_i)^{-1}\hspace*{-0.1em}>\hspace*{-0.1em}1$ with  $c_i(\varepsilon)\hspace*{-0.1em}:=\hspace*{-0.1em}(\rho_i\varepsilon)^{\smash{1-\rho_i'}}(\rho_i')^{-1}\hspace*{-0.1em}>\hspace*{-0.1em}0$ for all ${\varepsilon\hspace*{-0.1em}\in\hspace*{-0.1em} (0,\rho_i^{-1})}$~in~\eqref{eq:2.30.6}. In this way, using  $(a+b)^{\vphantom{1}\smash{p^-_i}}\hspace*{-0.1em}\leq\hspace*{-0.1em} 2^{p^+}(a^{\vphantom{1}\smash{p^-_i}}+b^{\vphantom{1}\smash{p^-_i}})$ and $a^{\vphantom{1}\smash{p^-_i}}\hspace*{-0.1em}\leq\hspace*{-0.1em} 2^{p^+}(1+a^{\vphantom{1}\smash{p^+_i}})$~for~all~${a,b\hspace*{-0.1em}\ge\hspace*{-0.1em} 0}$, we find that
	\begin{align}
		\rho_{p_i^+}(\bfu\chi_{B_i})&\leq c_i\varepsilon\big(\|\bfu\|_{\smash{L^{p_i^-}(B_i)}}+\|\bfD\bfu\|_{\smash{L^{p_i^-}(B_i)}}\big)^{p_i^-}+c_i(\varepsilon)\|\bfu\|_{L^2(B_i)}^{p_i^+(1-\theta_i)\rho_i'}\notag\\[-2pt]&\leq c_i\varepsilon 2^{p_i^+}\big(\rho_{p_i^-}(\bfu\chi_{B_i})+\rho_{p_i^-}(\bfD\bfu\chi_{B_i})\big)+c_i(\varepsilon)\|\bfu\|_{L^2(B_i)}^{p_i^+(1-\theta_i)\rho_i'}\label{eq:2.30.7}
		\\[-2pt]&\leq c_i\varepsilon 2^{2p_i^+}\big(\vert B_i\vert+\rho_{p_i^+}(\bfu\chi_{B_i})+\rho_{p_i^-}(\bfD\bfu\chi_{B_i})\big)+ c_i(\varepsilon)\|\bfu\|_{L^2(B_i)}^{p_i^+(1-\theta_i)\rho_i'}\,.\notag
	\end{align}
	We set $c_0:=\max_{i=1,\dots,m}{c_i}$ and $c_0(\varepsilon):=\max_{i=1,\dots,m}{c_i(\varepsilon)}$. Then, if we choose $\varepsilon:=\smash{2^{-2p^+-1}}c_0^{-1}$~and absorb $c_i\varepsilon 2^{2p_i^+}\rho_{\vphantom{1}\smash{p_i^+}}(\bfu\chi_{B_i})\hspace*{-0.1em}\leq \frac{1}{2}\rho_{\vphantom{1}\smash{p_i^+}}(\bfu\chi_{B_i})$ in the left-hand side in \eqref{eq:2.30.7}, we further infer from \eqref{eq:2.30.7} that
	\begin{align}
		\rho_{p_i^+}(\bfu\chi_{B_i})\leq 
		\vert \Omega\vert+\rho_{p_i^-}(\bfD\bfu\chi_{B_i})+2c_0(\varepsilon)\|\bfu\|_{L^2(B_i)}^{p_i^+(1-\theta_i)\rho_i'}\,.\label{eq:2.30.8}
	\end{align}
	We set $\gamma\hspace*{-0.1em}:=\hspace*{-0.1em}\max_{i=1,\dots,m}{p_i^+(1-\theta_i)\rho_i'}$ and use $\alpha^{p_i^+(1-\theta_i)\rho_i'}\hspace*{-0.1em}\leq\hspace*{-0.1em} 2^{\gamma}(1+\alpha^\gamma)$~for~all~${\alpha\hspace*{-0.1em}\ge \hspace*{-0.1em} 0}$,  $\rho_{p(\cdot)}(\bfu\chi_{B_i})\leq 2^{p^+}(\vert \Omega\vert+\rho_{\vphantom{1}\smash{p_i^+}}(\bfu\chi_{B_i}))$,  ${\rho_{\vphantom{1}\smash{p_i^-}}(\bfD\bfu\chi_{B_i})
		\leq 2^{p^+}(\vert\Omega\vert +\rho_{p(\cdot)}(\bfD\bfu\chi_{B_i}))}$  and $\rho_{p(\cdot)}(\bfD\bfu\chi_{B_i})\leq \rho_{p(\cdot)}(\bfD\bfu)$ in \eqref{eq:2.30.8}, to  arrive at
	\begin{align}
		\rho_{p(\cdot)}(\bfu\chi_{B_i})&
		\leq 2^{p^+}\big(\vert \Omega\vert+\rho_{p_i^+}(\bfu\chi_{B_i})\big)\notag\\[-4.5pt]&
		\leq 2^{p^+}\big(2\vert \Omega\vert+\rho_{p_i^-}(\bfD\bfu\chi_{B_i})+2c_0(\varepsilon)\|\bfu\|_{L^2(B_i)}^{p_i^+(1-\theta_i)\rho_i'}\big)\label{eq:2.30.9}\\[-1pt]&
		\leq 2^{p^+}\Big(2\vert \Omega\vert+2^{p^+}\big(\vert \Omega\vert +\rho_{p(\cdot)}(\bfD\bfu)\big)+c_0(\varepsilon)2^{\gamma+1}\big(1+\|\bfu\|_{L^2(\Omega'')}^{\gamma}\big)\Big)\,.\notag
	\end{align}
	If we sum up the inequalities \eqref{eq:2.30.2} and \eqref{eq:2.30.9} with respect~to~${j=1,\dots,m}$, we conclude that 
	\begin{align}
          \begin{aligned}
            \rho_{p(\cdot)}(\bfu \chi_{\Omega'})&\leq m\,
            2^{p^+}\Big (2\vert \Omega\vert + 2^{p^+}\big(\vert
            \Omega\vert + \rho_{p(\cdot)}(\bfD\bfu)\big) \Big)
            \\
            &\quad + m\,
            2^{p^+}\,c_0(\varepsilon)2^{\gamma+1}\big(1+\|\bfu\|_{L^2(\Omega'')}^{\gamma}\big)\,.
          \end{aligned}
\label{eq:2.30.10}
	\end{align}
	Moreover, since  $X^{p^-\!,p^-}_{\bD}\!(\Omega'')\!=\!W^{1,p^-}\!(\Omega'')$ with norm equivalence~by~Korn's~\mbox{second} inequality (cf. Lemma \ref{korn2}), where we, in particular, exploited that ${\partial \Omega''\in C^{0,1}}$, the Sobolev embedding theorem yields  $\smash{X^{p^-,p^-}_{\bD}(\Omega'')\hookrightarrow L^2(\Omega'')}$ since $p^-\ge\frac{2d}{d+2}$. Therefore, since $\smash{\XDp\hookrightarrow X^{p^-,p^-}_{\bD}(\Omega'')}$, we conclude from \eqref{eq:2.30.10} that there exists a constant $c(p,\Omega')>0$ such that for every $\bfu\in \smash{\mathcal{V}^{p^-,p(\cdot)}_{\bD}}$, it holds
	\begin{align*}
		\|\bfu\|_{L^{p(\cdot)}(\Omega')}\leq c(p,\Omega')\|\bfu\|_{\XDp}\,.
	\end{align*}

	\textbf{Step 2:}
	Let $\smash{\bfu\!\in\! \XDp}$ be arbitrary.
	\!Since $\smash{\mathcal{V}^{p^-,p(\cdot)}_{\bD}}$ is dense~in $\smash{\XDp}$, there is a sequence $(\bfu_n)_{n\in \mathbb{N}}\subseteq \smash{\mathcal{V}^{p^-,p(\cdot)}_{\bD}}$ such that $\bfu_n \to \bfu$ in $\smash{\XDp}$ $(n\to \infty)$. 
	According to Step $1$, there exists a constant ${c(p,\Omega')>0}$  such that for~every~${n\in \mathbb{N}}$
	\begin{align}
		\smash{\|\bfu_n\|_{L^{p(\cdot)}(\Omega')}\leq c(p,\Omega')\|\bfu_n\|_{\XDp}\,.}\label{eq:poincarep(x)folg1}
	\end{align}
	As a result, $(\bfu_n)_{n\in \mathbb{N}}$ is bounded in $ L^{p(\cdot)}(\Omega')$. Owing to the reflexivity of  $L^{p(\cdot)}(\Omega')$, there exists a cofinal subset $\Lambda\subseteq\mathbb{N}$ as well as a function $\tilde{\bfu}\in L^{p(\cdot)}(\Omega')$ such that $\bfu_n\rightharpoonup \tilde{\bfu}$ in $L^{p(\cdot)}(\Omega')$ $(\Lambda\ni n\to \infty)$. Thus, owing to the uniqueness of weak limits, we observe that $\bfu=\tilde{\bfu}\in L^{p(\cdot)}(\Omega')$. Finally, taking the limit inferior with respect to $n\to \infty$ in \eqref{eq:poincarep(x)folg1} proves that \eqref{eq:poincarep(x)} holds for every $\bfu\in \smash{\XDpo}$.
\end{proof}

\section*{References}

{\linespread{0.7}\selectfont

  \def\cprime{$'$} \def\cprime{$'$} \def\cprime{$'$}
\ifx\undefined\bysame
\newcommand{\bysame}{\leavevmode\hbox to3em{\hrulefill}\,}
\fi

\end{document}